\documentclass[iccmp]{ipbook}

\RequirePackage{amsthm,amsmath,amssymb,mathrsfs}
\RequirePackage{tabularx}
\usepackage{tikz,tcolorbox,rotating,float}
\usetikzlibrary{intersections}
\usetikzlibrary{arrows,shapes}
\usepackage{stackengine}

\startlocaldefs
\endlocaldefs

\newtheorem{The}{Theorem}[section]
 \newtheorem{Cor}[The]{Corollary}
 \newtheorem{Lem}[The]{Lemma}
 \newtheorem{Pro}[The]{Proposition}
 \theoremstyle{definition}
 \newtheorem{defn}[The]{Definition}
 \newtheorem{Prob}[The]{Problem}
 \newtheorem{Rem}[The]{Remark}
 \newtheorem{Ex}[The]{Example}
 \numberwithin{equation}{section}

\newcommand{\T}{\mathbb{T}}
\newcommand{\R}{\mathbb{R}}

\newcommand{\N}{\mathbb{N}}

\newcommand{\E}{\mathbb{E}}

\newcommand{\UC}{\mathrm{UC}\,}
\newcommand{\SCL}{\mathrm{SCL\,}}
\newcommand{\Lip}{\mathrm{Lip\,}}
\newcommand{\supp}{\mathrm{supp\,}}
\newcommand{\Cut}{\mathrm{Cut}\,}

\firstpage{1}
\lastpage{8}

\title[Generalized Hamiltonian gradient flow]{Recent progress in generalized Hamiltonian gradient flow: Singularities}

\author{Wei Cheng}
\address{School of Mathematics, Nanjing University, Nanjing 210093, China}
\email{chengwei@nju.edu.cn}
\thanks{}
\author{Jiahui Hong}
\address{School of Mathematics, Nanjing University of Aeronautics and Astronautics, Nanjing 211106, China}
\email{hongjiahui@nuaa.edu.cn}
\thanks{}


\begin{document}

\begin{abstract}
This paper is a survey of the generalized Hamiltonian gradient flow (GHGF) framework for Hamilton-Jacobi equations, with an emphasis on the propagation of singularities and its connections to weak KAM theory, optimal transport and mean field control. 

In addition to reviewing the main ideas and known results, we present two new contributions. First, we provide a variational construction of generalized characteristics via a minimizing movement scheme; by taking the weak limit of approximate solutions and using Young measure compactness, we show that the limiting curve satisfies the generalized characteristic differential inclusion. Second, we lift the forward dynamics of strict singular characteristics to a semi-flow on the space of trajectories and study its invariant probability measures. We prove that the only invariant measures of the GHGF semi-flow that attain the critical value \(c[H]\) are precisely the projected Mather measures, thereby giving a new dynamical characterization of Mather's minimal measures as well as Ma\~n\'e's critical value.

Finally, we discuss a number of open problems that arise from the GHGF perspective, including questions on uniqueness of strict singular characteristics, rectifiability of cut loci, stability under perturbations, contact Hamiltonian systems, vanishing noise limits, and extensions to non-convex or low-regularity Hamiltonians. These problems highlight the deeper connections between singular dynamics, ergodic theory, optimal transport, and geometric analysis, and indicate directions for future research.
\end{abstract}

\maketitle

\tableofcontents

\section{Introduction}

Let $M$ be a compact, connected, smooth manifold without boundary. Denote by $TM$ and $T^*M$ the tangent and cotangent bundles of $M$, respectively. A function $H:T^*M\to\mathbb{R}$ is called a \emph{Tonelli Hamiltonian} if it is of class $C^2$, $H_{pp}(x,p)>0$ for every $(x,p)\in T^*M$ and the map $p\mapsto H(x,p)$ is uniformly superlinear. The corresponding \emph{Tonelli Lagrangian} $L:TM\to\mathbb{R}$ is defined via the Fenchel transform.

\subsection{Viscosity solution and propagation of singularities}

We shall be concerned with both the evolutionary Hamilton--Jacobi equation
\begin{equation}\label{eq:HJe}\tag{HJ$_e$}
    \partial_t u(t,x)+H\bigl(x,D_x u(t,x)\bigr)=0,\qquad t>0,\;x\in M,
\end{equation}
and the stationary Hamilton--Jacobi equation
\begin{equation}\label{eq:HJ_wk}\tag{HJ$_s$}
    H\bigl(x,D\phi(x)\bigr)=c[H],\qquad x\in M.
\end{equation}
The constant $c[H]$ on the right‑hand side of \eqref{eq:HJ_wk} is the Ma\~n\'e critical value. Throughout the paper, unless stated otherwise, $c[H]=0$, $u$ and $\phi$ will denote viscosity solutions of \eqref{eq:HJe} and \eqref{eq:HJ_wk}, respectively, and $\phi$ will belong to the space $\mathrm{SCL}(M)$ of semiconcave functions with linear modulus. We recall that if $(M,g)$ is a Riemannian manifold, a continuous function $\phi:M\to\R$ is called semiconcave with linear modulus (with respect to the metric $g$) or $C$-semiconcave if
\begin{align*}
	\lambda\phi(x)+(1-\lambda)\phi(y)-\phi(\gamma(\lambda))\leqslant \frac C2\lambda(1-\lambda)d^2(x,y)
\end{align*}
for any $x,y\in M$, $\lambda\in[0,1]$ and a geodesic $\gamma:[0,1]\to M$ with $\gamma(0)=x$ and $\gamma(1)=y$, and $C\geqslant 0$ is called a semiconcavity constant for $\phi$.

It is well known that Hamilton-Jacobi equations typically do not admit globally defined smooth solutions, a phenomenon intrinsically linked to the emergence of singularities induced by the crossing or focusing of characteristic curves. The concept of viscosity solutions, introduced in the foundational contributions of Crandall and Lions~(\cite{Crandall_Lions1983}) and Crandall, Evans, and Lions~(\cite{Crandall_Evans_Lions1984}), furnishes the appropriate analytical framework for addressing questions of well-posedness for the evolutionary problem~\eqref{eq:HJe}. Comprehensive expositions of this theory for first‑order equations can be found in the monograph by Bardi and Capuzzo-Dolcetta~(\cite{Bardi_Capuzzo-Dolcetta1997}), while the extension to second‑order equations is treated in the work of Fleming and Soner~(\cite{Fleming_Soner_book2006}).

In the 1990s, Albert Fathi developed a theory for Hamilton-Jacobi equations, called the weak KAM theory (\cite{Fathi1997_1,Fathi1997_2,Fathi1998_1,Fathi1998_2}), which has deep connections with Mather's theory of area-preserving monotone twist maps (see for instance \cite{Mather_Forni1994} and the references therein) and with the Lagrangian dynamics of time-periodic Tonelli systems (see, for instance, the papers \cite{Mather1991,Mather1993,Mane1997,Mane1992,Contreras_Delgado_Iturriaga1997} and the books \cite{Contreras_Iturriaga_book,Sorrentino_book} with the references therein). As a global, non‑smooth variational framework, weak KAM theory transforms the study of Hamilton-Jacobi equations into a dynamical investigation of the associated Lagrangian minimizers. This viewpoint lies at the heart of modern analytic approaches to Hamiltonian systems. For further details on weak KAM theory, we refer the reader to Fathi’s original monograph \cite{Fathi_book} and to Tran’s recent book \cite{Tran_book}, which gives a more PDE‑oriented account.

The optimal regularity that one may expect for viscosity solutions of \eqref{eq:HJe} and \eqref{eq:HJ_wk} is local semiconcavity. Indeed, semiconcave functions were systematically employed in the study of well-posedness for \eqref{eq:HJe} even prior to the formal development of the theory of viscosity solutions (see, e.g., \cite{Douglis1961,Kruzkov1975,Krylov_book1987}). In contemporary mathematics, the notion of semiconcavity plays a fundamental role across a broad spectrum of disciplines. Prominent applications include optimal control theory and sensitivity analysis (\cite{Hrustalev1978,Cannarsa_Frankowska1991,Fleming_McEneaney2000,Rifford2000,Rifford2002,Bonnet_Frankowska2022}), nonsmooth and variational analysis (\cite{Rockafellar1982,Colombo_Marigonda2006}), as well as metric geometry (\cite{Petrunin2007}). Detailed and systematic treatments of semiconcave functions are provided in the monographs of Cannarsa and Sinestrari~(\cite{Cannarsa_Sinestrari_book}), Villani~(\cite{Villani_book2009}), and Ambrosio, Gigli, and Savar\'e~(\cite{Ambrosio_GigliNicola_Savare_book2008}).

The persistence of singularities (that is, the phenomenon whereby once a singularity appears it necessarily propagates forward in time up to $+\infty$) furnishes one striking manifestation of the intrinsic irreversibility of the evolutionary equation~\eqref{eq:HJe}. A further, closely related indication of this irreversibility is provided by the compactness of the evolution governed by the associated Lax-Oleinik semigroup; we refer the reader to~\cite{Ancona_Cannarsa_Nguyen2016_1,Ancona_Cannarsa_Nguyen2016_2} for detailed expositions of these properties.

To the best of our knowledge, the earliest investigation explicitly devoted to the singularities of viscosity solutions to the evolutionary Hamilton--Jacobi equation~\eqref{eq:HJe} is due to Cannarsa and Soner~(\cite{Cannarsa_Soner1987}). The subsequent identification of semiconcave functions as the natural regularity class for viscosity solutions of~\eqref{eq:HJe}~\cite{Cannarsa_Soner1989} paved the way for a systematic study of singularity propagation for general semiconcave functions, culminating in the results of~(\cite{Ambrosio_Cannarsa_Soner1993}). The specific propagation of singularities along Lipschitz arcs for semiconcave functions was first addressed in~\cite{Albano_Cannarsa1999} and later extended to the full setting of Hamilton-Jacobi equations in the influential work of Albano and Cannarsa~(\cite{Albano_Cannarsa2002}).

In~\cite{Albano_Cannarsa2002}, the authors introduced for the first time the fundamental concept of generalized characteristics for the Hamilton-Jacobi equation~\eqref{eq:HJe}, a notion that has since become a cornerstone for subsequent developments in the field. In the one‑dimensional setting, the underlying idea of generalized characteristics can also be traced back to the earlier work of Dafermos~(\cite{Dafermos1977}) on the Burgers equation.

In the present setting, where $\phi$ is a viscosity solution of the Hamilton--Jacobi equation~\eqref{eq:HJ_wk}, a Lipschitz curve $\gamma: [t_0,\infty) \to M$ with $t_0\in\R$ is termed a generalized characteristic for the pair $(\phi,H)$ if it satisfies the differential inclusion
\begin{equation}\label{eq:GC_intro}
	\dot{\gamma}(s) \in \operatorname{co} H_p\bigl(\gamma(s), D^+\phi(\gamma(s))\bigr), \qquad \text{a.e. } s \geqslant t_0,
\end{equation}
where $D^+\phi(x)$ denotes the set of superdifferential of a semicomcave function $\phi$ at $x$ and $\text{co}$ denotes the convex hull. It was established in~\cite{Albano_Cannarsa2002} that a generalized characteristic exists globally emanating from any initial point $x_0\in M$; moreover, if $x_0=\gamma(t_0)$ is a singular point, a point of nondifferentiablity of $\phi$, and $0 \notin \operatorname{co}\,H_p(x_0, D^+\phi(x_0))$, the singularity propagates locally along this curve (see~\cite{Yu2006} for subsequent refinements). 

Inspired by earlier contributions~(\cite{Bogaevsky2002,Bogaevski2006,Stromberg2013}), Khanin and Sobolevski (\cite{Khanin_Sobolevski2016}) established, under suitable additional assumptions on the initial data, the existence of solutions of the generalized characteristics differential inclusion \emph{without} the convex hull. Such curves are referred to as {strict singular characteristics} (or {broken characteristics} in the terminology of~\cite{Stromberg2013}). As pointed out by Khanin and Sobolevski, the notion of broken characteristics is intimately connected with problems arising in the study of certain stochastic Hamiltonian dynamical systems and the vanishing viscosity limit of Hamilton-Jacobi equations. We refer the reader to~\cite{Bogaevsky2002,Bogaevski2006,Khanin_Sobolevski2010,Khanin_Sobolevski2016} for comprehensive discussions of these intriguing connections. 

\subsection{Generalized Hamiltonian gradient flow}

A significant advance in the variational understanding of the various notions of singular characteristics found in the literature is the concept of {intrinsic singular characteristics}, first introduced in~\cite{Cannarsa_Cheng3} and further developed in~\cite{Cannarsa_Cheng_Fathi2017,Cannarsa_Cheng_Fathi2021}, where it led to important topological applications within weak KAM theory and Riemannian geometry. This intrinsic approach, combined with the notion of maximal slope curve originating from the Energy Dissipation Inequality (EDI) formalism of gradient flows, gives rise to a novel theory of GHGF~(\cite{CCHW2024}). The resulting framework has already demonstrated its considerable power in a variety of related fields, most notably optimal transport, mean field control, and geometric analysis.

Let $\phi$ be a semiconcave function on $M$ and let $H$ be a Tonelli Hamiltonian. 
A locally absolutely continuous curve $\gamma:\mathbb{R}\to M$ is called a {maximal slope curve for the pair $(\phi,H)$} in the EDI formalism, relative to a Borel measurable selection $\mathbf{p}(x)$ of the superdifferential $D^+\phi(x)$, if it satisfies
\begin{align*}
	\phi(\gamma(t)) - \phi(\gamma(0)) = \int_0^t \Bigl\{ L\bigl(\gamma(s), \dot{\gamma}(s)\bigr) + H\bigl(\gamma(s), \mathbf{p}(\gamma(s))\bigr) \Bigr\}\,ds, \ \forall t \in \mathbb{R}.
\end{align*}
%
A crucial observation is that the only meaningful selection for the theory is the {minimal energy selection}
\begin{align*}
	\mathbf{p}^{\#}_{\phi,H}(x) = \arg\min \bigl\{ H(x,p) : p \in D^{+}\phi(x) \bigr\}, \qquad x \in M.
\end{align*}
The map $x \mapsto \mathbf{p}^{\#}_{\phi,H}(x)$ is Borel measurable, and $x \mapsto H\bigl(x, \mathbf{p}^{\#}_{\phi,H}(x)\bigr)$ is lower semicontinuous. If $\phi$ is a viscosity solution of \eqref{eq:HJ_wk}, a curve $\gamma$ is called a {strict singular characteristic} for the pair $(\phi,H)$ if it satisfies
\begin{equation}\label{eq:SSC_intro}
    \dot{\gamma}(t) = H_p\bigl(\gamma(t), \mathbf{p}^{\#}_{\phi,H}(\gamma(t))\bigr), \qquad \text{a.e. } t \in \mathbb{R}.
\end{equation}
It was proved in~\cite{CCHW2024} that any strict singular characteristic also fulfills
\begin{equation}\label{eq:BC_intro}
    \dot{\gamma}^+(t) = H_p\bigl(\gamma(t), \mathbf{p}^{\#}_{\phi,H}(\gamma(t))\bigr), \qquad \forall t \in \mathbb{R},
\end{equation}
which implies, in particular, that the notions of {strict singular characteristics} and {broken characteristics} in \eqref{eq:BC_intro} coincide. 

To elucidate the intrinsic construction of strict singular characteristics, we draw upon a number of significant advances in Hamiltonian dynamical systems that have emerged over the past two decades. These include the Lasry-Lions regularization technique~(\cite{Bernard2007}), the concept of intrinsic singular characteristics, Arnaud's theorem~(\cite{Arnaud2011}), the commutator properties of Lax-Oleinik operators~(\cite{Cannarsa_Cheng_Hong2025}), together with the minimizing movement scheme originally developed in the abstract theory of gradient flows~(\cite{Ambrosio_GigliNicola_Savare_book2008,Ambrosio_Brue_Semola_book2021}) and optimal transport~(\cite{Bernard_Buffoni2006,Bernard_Buffoni2007a,Bernard_Buffoni2007b}). We present our first new result in Sect. \ref{subsec:weak_conv} and \ref{subsec:MM} on a detailed analysis showing that solutions of the generalized characteristic differential inclusion~\eqref{eq:GC_intro} can be understood as minimizing movements of the pair $(\phi,H)$.

In~\cite{CCHW2024}, the authors established the following global propagation result: let $H:T^*M\to\mathbb{R}$ be a Tonelli Hamiltonian and let $\phi$ be a weak KAM solution of the stationary Hamilton--Jacobi equation~\eqref{eq:HJ_wk}, if a curve $\gamma:[0,+\infty)\to M$ satisfies the generalized characteristic inclusion~\eqref{eq:GC_intro} and $\gamma(0)\in\operatorname{Cut}(\phi)$, the set of cut points or cut locus of $\phi$, then $\gamma(t)\in\operatorname{Cut}(\phi)$ for every $t\geqslant 0$. This allow us to study the singular dynamics of \eqref{eq:GC_intro} on $\operatorname{Cut}(\phi)$ for general Tonelli Hamiltonian, compared to the previous works on mechanical system (\cite{Cannarsa_Yu2009}).

According to the DiPerna-Lions theory, an ordinary differential equation such as~\eqref{eq:SSC_intro} naturally induces a corresponding mass transport at the level of probability measures. However, the direct application of this theory is precluded by the fact that the vector field in~\eqref{eq:SSC_intro} does not generate a regular Lagrangian flow, due to the crossing and focusing of forward characteristics. By exploiting the variational construction of strict singular characteristics outlined above, we nonetheless succeed in deriving the continuity equation associated with this irregular Lagrangian semiflow
\begin{align*}
	\begin{cases}
        \dfrac{d}{dt}\,\mu + \operatorname{div}\bigl( H_p(x,\mathbf{p}^{\#}_{\phi,H}(x)) \cdot \mu \bigr) = 0, \\[6pt]
        \mu_0 = \bar{\mu}.
    \end{cases}
\end{align*}
Moreover, along the resulting solution $\mu_t$ we obtain the monotonicity properties $\mu_{t_1}(\operatorname{Cut}(\phi)) \le \mu_{t_2}(\operatorname{Cut}(\phi))$ and $\mu_{t_1}(\overline{\operatorname{Sing}}(\phi)) \le \mu_{t_2}(\overline{\operatorname{Sing}}(\phi))$ for all $0 \le t_1 \le t_2$, which show that the mass concentrated on the cut locus and on the $C^{1,1}$ singular support is non‑decreasing in time (see \cite{CCHW2024}). 

A further extension to the framework of deterministic optimal transport was recently carried out in~\cite{CCSW2025}, which leads to a Hamilton-Jacobi equation
\begin{align*}
	\begin{cases}
        \partial_t U(t,\mu) + \displaystyle\int_{\mathbb{R}^{m}} H\bigl(x, \partial_{\mu}U(t,\mu)\bigr)\,d\mu = 0, & (t,\mu) \in \mathbb{R}^{+} \times \mathscr{P}_c(\mathbb{R}^{m}), \\[8pt]
        U(0,\mu) = \phi(\mu), & \mu \in \mathscr{P}_c(\mathbb{R}^{m}).
    \end{cases}
\end{align*}
on $\mathscr{P}_c(\mathbb{R}^{m})$, the space of all all Borel probability measures on $\R^m$ with compact support, where
\begin{align*}
	\Phi(\mu):=\int_{\mathbb{R}^{m}} \phi(x)\,d\mu, \qquad \mu \in \mathscr{P}(\mathbb{R}^{m}),
\end{align*}
is called the potential energy functional of a given semiconcave function $\phi$. A more refined analysis of this Hamilton-Jacobi equation on the Wasserstein space, the mass transport of measures driven by the GHGF, and the corresponding propagation of singularities was carried out in~\cite{CCSW2025}.

\subsection{Further new progress and open problems}
\label{subsec:intro3}

In the final part of this introductory section, we present a second new result concerning the invariant measures of the forward GHGF. More precisely, we prove that the only invariant measures of the semi-flow $(\mathcal{S}_{\phi,H}^{+}, \mathbf{P}^{t})$ that attain the Ma\~n\'e critical value $c[H]$ are exactly the projected Mather measures (Theorem~\ref{thm:Mather_max}). This provides a purely dynamical characterization of Mather's minimal measures within the framework of strict singular characteristics and highlights the distinguished role played by the Aubry set in the long-time behavior of the generalized flow.

To study the singular dynamics of \eqref{eq:GC_intro} on $\operatorname{Cut}(\phi)$, we should emphasize the essential difficulty here. There is no uniqueness result of the system of generalized characteristic inclusion~\eqref{eq:GC_intro}, even when considering the system of strict singular characteristics \eqref{eq:SSC_intro}. To our knowledge, the only model admits uniqueness of the solution of \eqref{eq:GC_intro} or \eqref{eq:SSC_intro} is when the Hamiltonian $H(x,p)$ is quadratic on $p$-variable, or we call the Riemannian case. For this reason, we have to borrow some idea of lifting method to consider such dynamical systems on the path space $C^0_+:=C^0([0,+\infty),M)$, endowed a distance $d_+(\gamma_1,\gamma_2)=\sup_{t\geqslant0}e^{-t}d(\gamma_1(t),\gamma_2(t))$, and a push-forward semigroup $\mathbf{P}^t:C^0_+\to C^0_+$, $t\geqslant0$ given by
\begin{align*}
	\mathbf{P}^t(\gamma)(s)=\gamma(t+s),\qquad\forall\gamma\in C^0_+.
\end{align*}
We define the evaluation map $\mathbf{V}^t:C^0_+\to M$, $t\geqslant0$ by $\mathbf{V}^t(\gamma)=\gamma(t)$. We regard $\mathcal{S}^+_{\phi,H}$, the set of all the solution of \eqref{eq:SSC_intro}, as a subspace of $(C^0_+,d_+)$. 

By classic Mather theory, for any $x\in\mathscr{M}(H)$, the (projected) Mather set, there exists a unique $V(x)\in T_xM$ such that $(x,V(x))\in\widehat{\mathscr{M}(H)}$, the Mather set in $TM$. Let $f_L:\mathscr{M}(H)\to\widehat{\mathscr{M}(H)}$, $f_L(x)=(x,V(x))$. Then $f_L$ is continuous. Define
\begin{align*}
	\gamma(x,t)=\pi\circ\Phi_L^t\circ f_L(x),\qquad x\in  \mathscr{M}(H),\ t\in[0,+\infty),
\end{align*}
and $\Gamma(x)(\cdot)=\gamma(x,\cdot)$ is a map from $\mathscr{M}(H)$ to $C^0_+$. The map $\Gamma$ is also continuous. For any weak KAM solution $\phi$ of \eqref{eq:HJ_wk} we have $\Gamma(\mathscr{M}(H))\subset\mathcal{S}^+_{\phi,H}$.

Then, we have the following new results on the invariant measures of the GHGF $(\mathcal{S}^+_{\phi,H},\mathbf{P}^t)$:
\begin{enumerate}
	\item We have the following characterization of Ma\~n\'e critical value:
	\begin{align*}
		c[H]=\max_{\phi\in\mathrm{SCL}\,(M)}\sup_{\mu\in\mathcal{IM}(\phi,H)}\int_{\mathcal{S}^+_{\phi,H}}H(\gamma(0),\mathbf{p}_{\phi,H}^\#(\gamma(0)))\ d\mu,
	\end{align*}
	where $\mathcal{IM}(\phi,H)$ is the set of all invariant measure of the semigroup $(\mathcal{S}^+_{\phi,H},\mathbf{P}^t)$.
	\item If $\phi$ is a weak KAM solution of \eqref{eq:HJ_wk}, $\mu_0\in\mathscr{M}_H$, then $\mu=\Gamma_\#\mu_0\in\mathcal{IM}(\phi,H)$ and $(\mathbf{V}^0)_\#\mu=\mu_0$. Moreover,
	\begin{equation}\label{eq:Mane_cri_intro}
		\int_{\mathcal{S}^+_{\phi,H}}H(\gamma(0),\mathbf{p}_{\phi,H}^\#(\gamma(0)))\ d\mu=c[H].
	\end{equation} 
	\item If $\phi\in\mathrm{SCL}\,(M)$ and $\mu\in\mathcal{IM}(\phi,H)$ satisfy \eqref{eq:Mane_cri_intro}, then $\mu_0=(\mathbf{V}^0)_\#\mu\in\mathscr{M}_H$ and $\mu=\Gamma_\#\mu_0$.
\end{enumerate}

Upon further reflection, it has become increasingly evident that the GHGF, together with the singular dynamics it induces on the cut locus, is intimately connected with several deep problems in Hamiltonian dynamical systems, mean field control, optimal transport, geometric analysis, as well as ergodic theory and statistical mechanics. In light of these connections, we devote the final section of the paper to collecting and discussing a number of open problems that, in our opinion, will help stimulate further progress in these directions. We refer the reader to a previous survey \cite{Cannarsa_Cheng2021a} and to the open problems discussed in \cite{Chen_Cheng2016,Cannarsa_Cheng_2020proceeding}.

The paper is organized as follows. Section~\ref{sec:generalized} is devoted to a systematic review of generalized characteristics and the propagation of singularities for semiconcave functions and weak KAM solutions. We recall the notion of intrinsic singular characteristics, the minimizing movement construction, and the weak convergence method via Young measures that leads to solutions of the generalized characteristic differential inclusion, which is a new result. In Section~\ref{sec:HGF}, we afford a more detailed review of the recent new theory of variational construction of strict singular characteristics. The main new result of the paper---a characterization of Mather measures as the unique invariant measures attaining the Ma\~n\'e critical value---is presented in this section (Theorem~\ref{thm:Mather_max}). Section~\ref{sec:Eulerian} shifts from the Lagrangian to the Eulerian viewpoint. We review the mass transport aspects of the GHGF. Finally, in Section~\ref{sec:open}, we collect and discuss a number of open problems. We believe that the GHGF framework developed in this paper provides a natural variational and dynamical language to address these challenges.

\medskip

\noindent\textbf{Acknowledgements.} Wei Cheng was partly supported by the National Natural Science Foundation of China (Grant No. 12231010). Jiahui Hong was partly supported by the National Natural Science Foundation of China (Grant No. 12501245). Both authors thank Piermarco Cannarsa for his contribution to the Young measure approach presented in Section~\ref{subsec:weak_conv}, which remains a keystone for further progress using the minimizing movement and EDI formalism of GHGF. The authors also thank Tianqi Shi for corrections to this paper and for useful suggestions.

\section{Variational construction of generalized characteristics}
\label{sec:generalized}

We say a Lipschitz curve $\mathbf{x}:[0,\infty)\to M$, $\mathbf{x}(0)=x\in M$, propagates the singularities of a semiconcave function $\phi$ \emph{locally} if exists $\delta>0$ depending on the initial point $x\in M$ such that $\mathbf{x}(t)\in\text{Sing}\,(\phi)$ for $t\in(0,\delta_x]$ provided $x\in\text{Sing}\,(\phi)$. We say such a curve $\mathbf{x}$ propagates the singularities of $\phi$ \emph{globally} if $\mathbf{x}$ propagates the singularities of $\phi$ on $(0,\delta_x]$ and $\delta_x$ is independent of $x$. If $\phi$ is a weak KAM solution, we raise similar notions of propagation of cut points (locally and globally).

Before proceeding further, it is necessary to distinguish several subsets of $M$ associated with a weak KAM solution $\phi$ of the stationary Hamilton--Jacobi equation~\eqref{eq:HJ_wk}, according to the regularity properties of $\phi$.

\begin{itemize}
    \item We define the \emph{cut time function} of $\phi$ by
    \begin{equation}\label{eq:cut_time}
    	\begin{split}
    		\tau_\phi(x):= &\,\sup \bigl\{ t \geqslant 0 : \exists\ \text{a $(\phi,H)$-calibrated curve}\\
    		& \gamma \in C^1([0, t], M)\ \text{s.t.}\ \gamma(0) = x\bigr\}.
    	\end{split}
    \end{equation}
    The set of all cut points of $\phi$ is denoted by $\operatorname{Cut}(\phi)$ and is called the \emph{cut locus} of $\phi$. $\operatorname{Cut}\,(\phi)=\{x\in M: \tau_\phi(x)=0\}$. We call $\mathcal{I}\,(\phi)=\{x\in M: \tau_\phi(x)=+\infty\}$ the \emph{projected Aubry set with respect to $\phi$}.
    \item A point $x\in M$ is called a \emph{regular point} of $\phi$ if it is not a cut point. The set of all regular points of $\phi$ is denoted by $\operatorname{Reg}(\phi)$ and is referred to as the \emph{regular set} of $\phi$.
    Then 
    \item The \emph{differentiability set} of $\phi$ is defined by $\operatorname{Diff}(\phi):=\{x\in M:\phi\text{ is differentiable at }x\}$, and the \emph{singular set} of $\phi$ is defined as $\operatorname{Sing}(\phi):= M\setminus\operatorname{Diff}(\phi)$.
    \item By Alexandrov's theorem, $\phi$ is twice differentiable almost everywhere. We denote by $\operatorname{Alex}(\phi)$ the \emph{Alexandrov set} of points in $M$ where $\phi$ admits a second order Taylor expansion (i.e., is twice differentiable in the Alexandrov sense). Each point in $\operatorname{Alex}(\phi)$ is termed an \emph{Alexandrov point}.
    \item The \emph{$C^{1,1}$ singular support} of $\phi$ is defined as $\operatorname{Singsupp}_{C^{1,1}}(\phi):=\overline{\operatorname{Sing}(\phi)}$.
\end{itemize}

The following inclusions among these sets are straightforward consequences of the definitions:
\begin{itemize}
    \item $\operatorname{Alex}(\phi)\subset\operatorname{Reg}(\phi)\subset\operatorname{Diff}(\phi)$;
    \item $\operatorname{Sing}(\phi)\subset\operatorname{Cut}(\phi)\subset\overline{\operatorname{Sing}(\phi)}=\operatorname{Singsupp}_{C^{1,1}}(\phi)$;
    \item $\operatorname{Alex}(\phi)\cap\operatorname{Cut}(\phi)=\varnothing$.
\end{itemize}

Let $H$ be a Tonelli Hamiltonian and let $A_t(x,y)$ be the associated fundamental solution given by
\begin{align*}
	A_t(x,y)=\inf_{\xi\in\Gamma^t_{x,y}}\int^t_0L(\xi(s),\dot{\xi}(s))\ ds,\quad t>0,\ x,y\in M,
\end{align*}
with $\Gamma^t_{x,y}$ the family of absolutely continuous curve $\xi:[0,t]\to M$ with $\xi(0)=x$, $\xi(t)=y$. For any continuous real-valued function $\phi$ on $M$ and any Tonelli Lagrangian $L$, as in weak KAM theory, we define the Lax-Oleinik semigroups $\{T^{\pm}_t\}_{t\geqslant0}$ as operators
\begin{align*}
	\begin{split}
		T^+_t\phi(x)=&\,\sup_{y\in M}\{\phi(y)-A_t(x,y)\},\\
	T^-_t\phi(x)=&\,\inf_{y\in M}\{\phi(y)+A_t(y,x)\},
	\end{split}
	\quad x\in M.
\end{align*}
We call $\phi$ a \emph{weak KAM solution} of \eqref{eq:HJ_wk} if $T^-_t\phi+c[H]t=\phi$ for all $t\geqslant 0$. 

\subsection{Generalized characteristics}

To explain the application of the generalized characteristics to the problem of propagation of singularities of semiconcave functions and especially the viscosity solutions, we should recall various type singular characteristics in the literature. 

The notion of generalized characteristics, which plays an important role to describe the evolution of singularities of viscosity solutions, was first introduced in \cite{Albano_Cannarsa2002} in the frame of Hamilton-Jacobi equations ($M=\R^n$). In the case when $H$ is a Tonelli Hamiltonian, a Lipschitz curve $\mathbf{x}:[0,\infty)\to\R^n$ satisfying the following differential inclusion is called a \emph{generalized characteristic} from $x\in\R^n$:
\begin{equation}\label{eq:GC}\tag{GC}
	\begin{cases}
		\dot{\mathbf{x}}(t)\in\text{co}\,H_p(\mathbf{x}(t),D^+\phi(\mathbf{x}(t))),\quad a.e.\  t\in[0,+\infty)\\
		\mathbf{x}(0)=x,
	\end{cases}
\end{equation}
for any $\phi\in\text{\rm SCL}\,(M)$. We call such a curve $\mathbf{x}$ a \emph{strict generalized characteristic} (\cite{Cannarsa_Cheng2021b}) if we rule out the convex hull in \eqref{eq:GC}. The authors proved (\cite{Albano_Cannarsa2002}) if $\phi$ is a viscosity solution of the Hamilton-Jacobi equation
\begin{align*}
	H(x,D\phi(x))=0,\qquad x\in\Omega,
\end{align*}
where $\Omega\subset\R^n$ is an open set, then there exists a generalized characteristic $\gamma$ and $\delta>0$ such that $\gamma(t)\in\text{Sing}\,(\phi)$ for $t\in[0,\delta]$ provided $x\in \text{Sing}\,(\phi)$. 
Approximation method by standard convolution used in \cite{Yu2006,Cannarsa_Yu2009} was also applied to obtain further extended properties for any pair $(\phi,H)$ with $H$ a $C^1$ and strictly convex Hamiltonian, and $\phi$ a semiconcave function. 

\begin{figure}
	\begin{center}
\begin{tikzpicture}
\clip (-2.5,-2.5) rectangle (2.5,2.5);
\draw [red][name path=curve 1] (-2,-2) .. controls (-1,.5) .. (2,2);
\draw (-2,-2) node[left] {$x$};
\draw (2,2) node[right] {$y_t$};
\draw (-1.2,0) node[left][red] {$\displaystyle \mathbf{y}_x(s)$};
\draw [blue][name path=curve 2] (-2,-2) .. controls (.8,-.5) .. (2,2);
\fill [name intersections={of=curve 1 and curve 2, by={a,b}}] 
        (a) circle (2pt)
        (b) circle (2pt);
\draw (1.2,0) node[right][blue] {$\displaystyle \xi_t(s)$};
\draw [blue][name path=curve 3] (-2,-2) .. controls (.6,-.4) .. (0.9,1.45);  
\fill [name intersections={of=curve 1 and curve 3, by={a,c}}] 
        (a) circle (2pt)
        (c) circle (2pt);
        \draw (0.9,1.45) node[above] {$y_{t'}$};
\draw [blue][name path=curve 4] (-2,-2) .. controls (0.35,-.24) .. (0.4,1.24); 
\fill [name intersections={of=curve 1 and curve 4, by={a,d}}] 
        (a) circle (2pt)
        (d) circle (2pt);
\draw (0.4,1.24) node[above] {$y_{t''}$};
\end{tikzpicture}
\end{center}
\caption{Construction of intrinsic singular characteristics}
\end{figure}
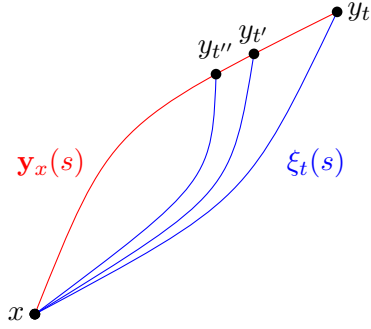

The notion of intrinsic singular characteristics comes from the paper \cite{Cannarsa_Cheng3} and \cite{Cannarsa_Cheng_Fathi2017,Cannarsa_Cheng_Fathi2021}. Let $\phi$ be a weak KAM solution of \eqref{eq:HJ_wk}. For any $x\in M$ we define a curve $\mathbf{y}_x:[0,\tau(\phi)]\to M$ with $\tau(\phi)>0$ determined below
\begin{equation}\label{eq:ISC}\tag{ISC}
	\mathbf{y}_x(t):=
	\begin{cases}
		x,& t=0,\\
		\arg\max\{\phi(y)-A_t(x,y): y\in M\},& t\in(0,\tau(\phi)].
	\end{cases}
\end{equation}
The construction of the curve $\mathbf{y}_x(t)$ is explained as follows
\begin{itemize}
	\item In the definition $T^+_t\phi=\sup\{\phi(y)-A_t(x,y): y\in M\}$, the supremum is achieved for any $t>0$. Indeed, there exists a constant $\lambda_0>0$ depending on $\text{Lip}\,(\phi)$ such that if $y$ is a maximizer of $\phi(\cdot)-A_t(x,\cdot)$ then $y\subset B(x,\lambda_0 t)$.
	\item Taking $\lambda=\lambda_0+1$ and applying some regularity result for $A_t(x,y)$ (see, for instance, \cite{Cannarsa_Cheng3}), we conclude that, there exists $t_0>0$ such that for any $t\in(0,t_0)$ and $x\in M$, the functions $y\mapsto A_t(x,y)$, $y\in B(x,\lambda t)$ is of class $C^2$ and convex with constant $C_2/t$.
	\item Since $\phi$ is semiconcave with constant $C_1$, the function $\phi(\cdot)-A_t(x,\cdot)$ is strictly concave provided $C_1-C_2/t<0$. Take $0<\tau(\phi)\leqslant t_0$ such that $C_1-C_2/\tau(\phi)<0$.
	\item Since $T^+\phi$ is naturally semiconvex, we conclude that if $t\in(0,\tau(\phi))$, then $T^+_t\phi\in C^{1,1}(M)$ and the maps $x\mapsto A_t(x,y)$, $x\in B(y,\lambda t)$ and $y\mapsto A_t(x,y)$, $y\in B(x,\lambda t)$ are convex with constant $C_2/t$ by regularity property of fundamental solution for short time.
\end{itemize}
Observe that the construction in \eqref{eq:ISC} can be extended to $[0,+\infty)$, by an \emph{Euler segment method}, since $\tau(\phi)$ is independent of $x$. We call $\mathbf{y}_x$ an \emph{intrinsic singular characteristic} from $x$ (with time step $\tau=\tau(\phi)$). A very important observation from this construction is the following global propagation result: $\mathbf{y}_x(t)\in\text{Sing}\,(\phi)$ for all $t>0$ provided that $x\in\text{Cut}\,(\phi)$. This construction has a quite natural philosophy to consider the interaction of \emph{forward-backward evolution with irreversibility}. However, $\mathbf{y}_x$ \emph{does not satisfies semigroup properties!} Thus a natural idea is to let the time step $\tau(\phi)\to 0^+$.

Invoking a celebrated paper by Khanin and Sobolevski (\cite{Khanin_Sobolevski2016}, see also \cite{Stromberg_Ahmadzadeh2014}), we consider the equation
\begin{equation}\label{eq:SGC2}\tag{BC}
	\begin{cases}
		\dot{\gamma}^+(t)=H_p(\gamma(t),\mathbf{p}^{\#}_{\phi,H}(\gamma(t))),\qquad\forall t\in[0,+\infty),\\
		\gamma(0)=x.
	\end{cases}
\end{equation}
A solution of \eqref{eq:SGC2} is called a \emph{broken characteristic} in the literature (\cite{Stromberg2013}). For the local existence of broken characteristics and strict generalized characteristics, there are several proofs (\cite{Khanin_Sobolevski2016,Cannarsa_Cheng2021a,Cheng_Hong2022a,Cheng_Hong2023}). The original one \cite{Khanin_Sobolevski2016} used vanishing viscosity approximation methods applying to  evolutionary Hamilton-Jacobi equation, and later \cite{Cannarsa_Cheng2021a} used a standard approximation by convolution for the static equation. Local existence results for \eqref{eq:SGC2} for viscosity solutions to evolutionary Hamilton-Jacobi equations with smooth initial data and weak KAM solution in 2D were obtained (\cite{Cheng_Hong2022a,Cheng_Hong2023}) based on the analysis of the underlying characteristic systems.

\subsection{Birth of singularities: a geometric intuition}

\begin{figure}[b]
\begin{tikzpicture} [scale=1.5]
\draw[->] (-2.6,0) -- (-0.5,0) node[right]{$x$};
\draw[->] (1.4,0) -- (3.5,0) node[right]{$x$};
\draw[->] (-2.5,-0.1) -- (-2.5,2) node[right]{$p=DT_t^+\phi(x)$};  
\draw[->] (1.5,-0.1) -- (1.5,2) node[right]{$p=D^+\phi(x)$};
\draw [domain=-2.4:-0.8,black,thick] plot(\x,{3.15476*\x*\x*\x+14.7679*\x*\x+21.4113*\x+10.508});
\draw[domain=1.6:2.3,black,thick] plot (\x,{-1.02*(\x-2.3)*(\x-2.3)+1.5});
\draw[red,thick] (2.3,1.5) .. controls (2.3,1.1).. (2.3,0.4);
\draw[domain=2.3:3.3,black,thick] plot (\x,{0.494*(\x-2.3)*(\x-2.3)+0.4});
\draw[red] (2.3,1) node[right] {$\scriptstyle D^+\phi(x_0)$};
\filldraw[red] (2.3,0) circle (0.5pt) node[anchor= south west] {$\scriptstyle x_0$};
\draw[red,dashed] (2.3,0.4) -- (2.3,0);
\draw[->] (-0.2,1) .. controls (0.5,1.1) .. (1.2,1);
\draw (0.5,1.1) node[above]{$\Phi_H^t$};
\end{tikzpicture}
\caption{Illustration of Arnaud's Theorem}
\end{figure}
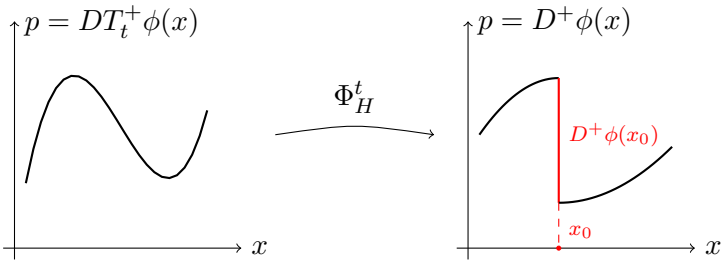

Let $\phi\in\text{\rm SCL}\,(M)$, the space of semiconcave functions with linear modulus on $M$,  and
\begin{align*}
	\text{\rm graph}\,(D^+\phi)=&\,\{(x,p): x\in M, p\in D^+\phi(x)\subset T^*_xM\}.
\end{align*}

\begin{The}[\cite{Bernard2007,Arnaud2011,Bernard2012}]\label{thm:T^-T^+1}
Let $\phi\in\text{\rm SCL}\,(M)$ and $t_0>0$ such that $T^+_{t_0}\phi\in C^{1,1}(M)$.
\begin{enumerate}
	\item $T^+_t\phi\in C^{1,1}(M)$ for all $t\in(0,t_0]$ and $T^-_{t}\circ T^+_{t}\phi=\phi$ for all $t\in[0,t_0]$.
	\item $T^+_t\phi=T^-_{t_0-t}\circ T^+_{t_0}\phi$ for all $t\in[0,t_0]$.
	\item \label{itm:Arnaud}We have
	\begin{align*}
		\text{\rm graph}\,(DT^+_t\phi)=\Phi_H^{-t}(\text{\rm graph}\,(D^+\phi)),\qquad t\in(0,t_0].
	\end{align*}
\end{enumerate} 
\end{The}

Such Lasry-Lions regularization type results in Theorem \ref{thm:T^-T^+1} is essential for our analysis of the birth of singularities and irreversibility. Similar work as statement (\ref{itm:Arnaud}) in Theorem \ref{thm:T^-T^+1}, known as Arnaud's theorem, was also obtained in \cite{Bianchini_Tonon2012} on the study of the SBV regularity of viscosity solutions in Euclidean case, and was rediscovered in \cite{Cannarsa_Cheng_Hong2025} from control aspects of Lax-Oleinik commutators.

From Arnaud theorem, an infinitesimal generation of singularities can be understood as certain projective singularities. 
Therefore, in it natural to send the time step to $0$ in \eqref{eq:ISC}.

\begin{Pro}[\cite{CCHW2024}]\label{pro:ISC_t}
Fix $x\in M$, $t\in [0,\tau(\phi)]$ and $\mathbf{y}_x$ is defined in \eqref{eq:ISC}. Let $\xi_t\in\Gamma^t_{x,\mathbf{y}_x(t)}$ be the minimal curve for $A_t(x,\mathbf{y}_x(t))$, then
\begin{align*}
	\xi_t(s)=\pi_x\Phi^s_H(x,DT^+_t\phi(x))=\pi_x(\xi_t(s),DT^+_{t-s}\phi(\xi_t(s))),\quad \forall s\in[0,t),
\end{align*}
where $\Phi^s_H$ denotes the associated Hamiltonian flow. In local coordinates, $\xi_t$ satisfies differential equation
\begin{equation}\label{eq:vector_field1}
	\begin{cases}
		\dot{\xi}_t(s)=H_p(\xi_t(s),DT^+_{t-s}\phi(\xi_t(s))),& s\in[0,t),\\
		\xi_t(0)=x.
	\end{cases}
\end{equation}
Moreover, there exists a constant $C>0$ such that
	\begin{align*}
		\|\mathbf{y}_x-\xi_t\|_{\infty}\leqslant Ct.
	\end{align*}	
\end{Pro}

Suppose $\phi\in\text{\rm SCL}\,(M)$ and $H:T^*M\to\R$ is a Tonelli Hamiltonian. Invoking Proposition \ref{pro:ISC_t}, we observe that for given $t\in[0,\tau(\phi)]$,
\begin{equation}\label{eq:vf_W}
	W(s,x)=H_p(x,DT^+_{t-s}\phi(x)),\qquad (s,x)\in [0,t)\times M
\end{equation}
defines a time-dependent vector field on $[0,t)\times M$. Now, we try to extend the vector field $W$ to $[0,\infty)\times M$. 

Let $\Delta$ be a partition of $[0,\infty)$ with the partition points $=\{\tau_i\}$, an increasing sequence of positive real numbers, with $|\Delta|\leqslant\tau(\phi)$, where $|\Delta|=\max\{\tau_i-\tau_{i-1}\}$ is the width of the partition. For any $t\in[0,\infty)$, we let $\tau_{\Delta}(t)=\inf\{\tau_i|\,\tau_i>t\}$. Then one can define a vector field $W_{\Delta}(t,x)$ on $[0,\infty)\times M$ as
\begin{equation}\label{eq:vf_W_Delta}
	W_{\Delta}(t,x)=H_p(x,DT^+_{\tau_{\Delta}(t)-t}\phi(x)),\qquad t\in[0,\infty), x\in M.
\end{equation}
Notice that the vector fields $W_{\Delta}$ is uniformly bounded for any partition $\Delta$, and $W_{\Delta}$ is Lipschitz continuous on each $[\tau_{i-1},\tau_i-\varepsilon]\times M$ for any small $\varepsilon>0$. 

We consider the differential equation
\begin{equation}\label{eq:ODE_W_Delta}
	\begin{cases}
		\dot{\gamma}(t)=W_{\Delta}(t,\gamma(t)),\qquad t\in[0,\infty),\\
		\gamma(0)=x\in M.
	\end{cases}
\end{equation}
Equation \eqref{eq:ODE_W_Delta} admits a piecewise $C^1$ solution $\gamma:[0,\infty)\to M$, which is of $C^1$ class on each partition interval. Moreover, the solution of \eqref{eq:ODE_W_Delta} is unique on $[0,\infty)$, by Cauchy-Lipschitz theorem.

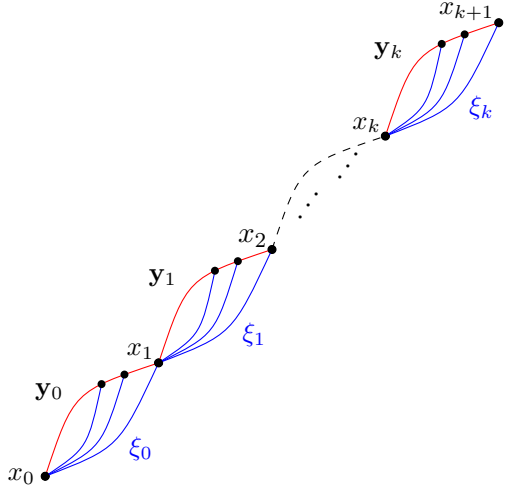
\begin{figure}
 \begin{center}
  \begin{tikzpicture}[scale=0.75]
   \clip (-6,-5) rectangle (6,5);
   \draw [red] [name path=curve 1] (-4,-4) ..  controls (-3.5,-2.5) .. (-2,-2);
      \draw (-4,-4) node[left] {\small{$x_0$}};
      \filldraw (-4,-4) circle (2pt);
      \draw (-1.9,-1.8) node[left] {\small{$x_1$}};
      \filldraw (-2,-2) circle (2pt);
      \draw (-3.5,-2.5) node[left][black]{\small{$\mathbf{y}_0$}};
      \draw [blue][name path= curve 11] (-4,-4) .. controls (-2.7,-3.5) .. (-2,-2);
      \draw (-2.7,-3.5) node[right][blue]{\small{$\xi_0$}};
      \fill [name intersections={of=curve 1 and curve 11, by={a,b}}] 
               (a) circle (2pt)
               (b) circle (2pt);
            \draw [blue][name path= curve 12] (-4,-4) .. controls (-3,-3.5) .. (-2.6,-2.2);
            \fill [name intersections={of=curve 1 and curve 12, by={a,c}}] 
               (a) circle (2pt)
               (c) circle (2pt);
            \draw [blue][name path= curve 13] (-4,-4) .. controls (-3.25,-3.5) .. (-3,-2.35);
            \fill [name intersections={of=curve 1 and curve 13, by={a,c}}] 
               (a) circle (2pt)
               (c) circle (2pt);
             \draw [red] [name path=curve 2] (-2,-2) ..  controls (-1.5,-.5) .. (0,0); 
              \draw (0.1,0.2) node[left] {$x_2$};
              \filldraw (0,0) circle (2pt);
            \draw (-1.5,-.5) node[left][black]{\small{$\mathbf{y}_1$}};
      \draw [blue][name path= curve 21] (-2,-2) .. controls (-.7,-1.5) .. (0,0);
      \draw (-.7,-1.5) node[right][blue]{\small{$\xi_1$}};
      \fill [name intersections={of=curve 2 and curve 21, by={a,b}}] 
               (a) circle (2pt)
               (b) circle (2pt);
            \draw [blue][name path= curve 22] (-2,-2) .. controls (-1,-1.5) .. (-.6,-.2);
            \fill [name intersections={of=curve 2 and curve 22, by={a,c}}] 
               (a) circle (2pt)
               (c) circle (2pt);
            \draw [blue][name path= curve 23] (-2,-2) .. controls (-1.25,-1.5) .. (-1,-.35);
            \fill [name intersections={of=curve 2 and curve 23, by={a,c}}] 
               (a) circle (2pt)
               (c) circle (2pt);
            \draw [dashed] (0,0) .. controls (0.5,1.5) ..  (2,2);
            \draw (2.1,2.2) node[left] {\small{$x_{k}$}};
            \filldraw (2,2) circle (2pt);
            \draw (0.5,0.3) node[above]{\begin{sideways}$\ddots$\end{sideways}};
            \draw (1.2,1) node[above]{\begin{sideways}$\ddots$\end{sideways}};
            \draw [red] [name path=curve 3] (2,2) ..  controls (2.5,3.5) .. (4,4);
            \draw (4.1,4.2) node[left] {\small{$x_{k+1}$}};
              \filldraw (4,4) circle (2pt);
            \draw (2.5,3.5) node[left][black]{\small{$\mathbf{y}_k$}};
      \draw [blue][name path= curve 31] (2,2) .. controls (3.3,2.5) .. (4,4);
      \draw (3.3,2.5) node[right][blue]{\small{$\xi_k$}};
      \fill [name intersections={of=curve 3 and curve 31, by={a,b}}] 
               (a) circle (2pt)
               (b) circle (2pt);
            \draw [blue][name path= curve 32] (2,2) .. controls (3,2.5) .. (3.4,3.8);
            \fill [name intersections={of=curve 3 and curve 32, by={a,c}}] 
               (a) circle (2pt)
               (c) circle (2pt);
            \draw [blue][name path= curve 33] (2,2) .. controls (2.75,2.5) .. (3,3.65);
            \fill [name intersections={of=curve 3 and curve 33, by={a,c}}] 
               (a) circle (2pt)
               (c) circle (2pt);
  \end{tikzpicture}
 \end{center}
 \caption{minimizing movement from $x_0$.}
\end{figure}

\subsection{Weak convergence of the minimizing movement}
\label{subsec:weak_conv}

Now, fix $T>0$ and pick up any sequence of partitions $\{\Delta_m\}$ such that $\lim_{m\to\infty}|\Delta_m|=0$, and for each $m\in\N$, denote by $\mathbf{x}_m$ the solution of the equation
\begin{equation}\label{eq:ODE}
	\begin{cases}
		\dot{\mathbf{x}}_m(t)=H_p(\mathbf{x}_m(t),\mathbf{p}_m(t)),\qquad t\in[0,T],\\
		\mathbf{x}_m(0)=x,
	\end{cases}
\end{equation}
where $\mathbf{p}_m(t):=DT^+_{\tau_{i_{\Delta_m}(t)}-t}\phi(\mathbf{x}_m(t))$. Since the family of solutions $\{\mathbf{x}_m\}$ is equi-Lipschitz and bounded on $[0,T]$, by taking a subsequence if necessary, a standard application of Ascoli-Arzela theorem gives us a limiting Lipschitz curve $\mathbf{x}:[0,T]\to M$. It is natural to ask if the Lipschitz curve $\mathbf{x}_m$, as $m\to\infty$, tends to a solution of \eqref{eq:GC}.

To answer this question, we will use the method of weak convergence and \emph{Young measures} (see \cite{Yu2006} for another approach using weak convergence idea for singular characteristics). Let us first introduce the basic facts from Young measure named after L. C. Young (see \cite{Young_book}). The following result is standard.

\begin{Pro}[\cite{Evans_book1990}]\label{pro:Young_measure}
Let $m,n\in\N$, $U\subset\R^n$ be a bounded open subset and $\{f_k\}_{k\in\N}$ be a sequence in $L^{\infty}(U,\R^m)$. Then, there exists a subsequence $\{f_{k_j}\}$ and for almost all $x\in U$ a Borel probability measure $\nu_x$ on $\R^m$ such that for each $F\in C(U\times\R^m)$ we have
\begin{align*}
	F(\cdot,f_{k_j})\ \mathrel{\ensurestackMath{\stackon[1pt]{\rightharpoonup}{\scriptstyle\ast}}}\ \overline{F}\quad \text{in}\ L^{\infty}(U),
\end{align*}
where
\begin{align*}
	\overline{F}(x)=:\int_{\R^m}F(x,y) d\nu_x(y),\quad a.e., x\in U.
\end{align*}
\end{Pro}

\begin{Rem}
If $\{K(x)\}_{x\in U}$ is a measurable family of convex and compact sets such that
\begin{equation}\label{eq:upper_semicontinuity}
	\lim_{k\to\infty}d_{K(x)}(f_k(x))=0,\quad a.e., x\in U,
\end{equation}
then $\text{supp}\, \nu_x\subset K(x)$ for almost all $x\in U$. Indeed, choosing $F(x,y)=d_{K(x)}(y)$ in Proposition \ref{pro:Young_measure}, we obtain that for any $\Psi\in L^1(U)$,
\begin{align*}
	0=\lim_{k\to\infty}\int_U\langle d_{K(x)}(f_k(x)),\Psi(x)\rangle\ dx=\int_U\Psi(x)\left(\int_{\R^m}d_{K(x)}(y)\ d\nu_x(y)\right)\ dx.
\end{align*}
It follows that $\text{supp}\, \nu_x\subset K(x)$ for almost all $x\in U$.
\end{Rem}

\begin{Lem}\label{lem:weak}
Let $\mathbf{x}_{m}$ be the solution for \eqref{eq:ODE} converging to $\mathbf{x}$ uniformly on $[0,T]$ as $m\to\infty$, then
\begin{align*}
	\lim_{m\to\infty}d_{D^+u(\mathbf{x}(t))}(\mathbf{p}_{m}(t))=0,\qquad t\in[0,T].
\end{align*}
\end{Lem}

\begin{proof}
This is direct consequence of the upper semi-continuity of the set-valued map $(t,x)\rightrightarrows D^+u(t,x)$ with $u(t,x)=T^+_t\phi(x)$.
\end{proof}

\begin{The}\label{thm:Young}
Suppose $\{\mathbf{x}_m\}$ is a sequence of solution of \eqref{eq:ODE} and $\mathbf{x}_m$ converges to $\mathbf{x}$ uniformly on $[0,T]$ as $m\to\infty$.
\begin{enumerate}
	\item There exists a family of Borel probability measure $\{\nu_t\}_{t\in[0,T]}$ such that $\mbox{\rm supp}\,\nu_t\subset D^+\phi(\mathbf{x}(t))$ and there exists a subsequence $\mathbf{p}_{m_k}$ for almost all $t\in[0,T]$ such that
	\begin{equation}\label{eq:weak_convergence}
	H_p(\mathbf{x}_{m_k},\mathbf{p}_{m_k})\mathrel{\ensurestackMath{\stackon[1pt]{\rightharpoonup}{\scriptstyle\ast}}}\ \overline{H}_p\quad \text{in}\ L^{\infty}([0,T])
	\end{equation}
	as $m\to\infty$, where $\overline{H}_p(t)=\int_{\R^n}H_p(\mathbf{x}(t),p)\ d\nu_t(p)$ for $t\in[0,T]$.
	\item The curve $\mathbf{x}$ satisfies equation
	\begin{equation}\label{eq:gc2}
	    \dot{\mathbf{x}}(t)=\overline{H}_p(t),\quad a.e.\ t\in[0,T].
	\end{equation}
	\item $\mathbf{x}$ satisfies the differential inclusion \eqref{eq:GC} and it is a generalized characteristic.
\end{enumerate}
\end{The}

\begin{proof}
Since $\mathbf{x}_m(\cdot)$ satisfies differential equation \eqref{eq:ODE}, without loss of generality we also suppose that $\mathbf{x}_m$ converges to $\mathbf{x}$ uniformly on $[0,T]$ as $m\to\infty$. Thus, by Proposition \ref{pro:Young_measure}, there exists a subsequence $\{\mathbf{p}_{m_k}\}$ and Young measure $\{\nu_t\}$, for almost all $t\in[0,T]$ such that
\begin{align*}
	H_p(\mathbf{x}_{m_k},\mathbf{p}_{m_k})\mathrel{\ensurestackMath{\stackon[1pt]{\rightharpoonup}{\scriptstyle\ast}}}\ \overline{H}_p\quad \text{in}\ L^{\infty}([0,T]).
\end{align*}
Moreover, we have $\lim_{k\to\infty}d_{D^+(\mathbf{x}(t))}(\mathbf{p}_{m_k}(t)))=0$ by Lemma \ref{lem:weak}. Then condition \eqref{eq:upper_semicontinuity} implies $\text{supp}\,\nu_t\subset D^+\phi(\mathbf{x}(t))$ for almost all $t\in[0,T]$. This completes the proof of (1). 

The proofs of (2) and (3) are direct consequences of (1), since $\{\mathbf{x}_m\}$ converges to $\mathbf{x}$ uniformly on $[0,T]$, and $\nu_t$ is a family of Borel probability measure. 
\end{proof}

Invoking our previous construction, for any partition $\Delta$ of $[0,T]$, we define a curve $\mathbf{y}_\Delta:[0,T]\to M$ inductively by
\begin{align*}
	\mathbf{y}_\Delta(t):=\arg\max\{\phi(z)-A_{t-\tau_{i_{\Delta}(t)-1}}(\mathbf{y}(\tau_{i_{\Delta}(t)-1}),z): z\in M\}
\end{align*}
and $\mathbf{y}_k:=\mathbf{y}_{\Delta_k}$, $k\in\N$, with $|\Delta_k|<\tau(\phi)$. Directly from Proposition 3.4 in \cite{Cannarsa_Cheng3} and Proposition \ref{pro:ISC_t}, the sequence $\{\mathbf{y}_k\}$ is equi-Lipschitz and there exists $C>0$ such that for each $k\in\N$,
\begin{align*}
	\|\mathbf{y}_k-\mathbf{x}_k\|_{\infty}\leqslant C|\Delta_k|.
\end{align*}

\begin{Cor}
Any limiting curve $\mathbf{x}$ in Theorem \ref{thm:Young} is Lipschitz and it a generalized characteristic from $x$.
\end{Cor}

\subsection{Generalized characteristics and minimizing movements}
\label{subsec:MM}

One can simplify our previous construction by using a uniform partition. Let $x\in M$ and $\tau\in(0,\tau(\phi))$. We define inductively a sequence $\{z_k\}$ by $z_0=x$ and
\begin{align*}
	z_{k+1}=\arg\max\{\phi(y)-A_\tau(x,y), y\in M\},\quad k\in\N.
\end{align*}
The sequence $\{z_k\}$ is well defined by the discussion in the last subsection. Now, for any $\tau\in(0,\tau_3(\phi)$ we define a curve $\mathbf{z}^{\tau}:[0,+\infty)\to M$ by
\begin{align*}
	\mathbf{z}^\tau(t)=z_{\lfloor\frac t{\tau}\rfloor},\quad s\in[0,+\infty).
\end{align*}
It is clear that each $\mathbf{z}^{\tau}$ is a piecewise constant curve. As an analogy with classic notion of minimizing movements, a \emph{minimizing movements of $(\phi,H)$ from $x$ in positive direction} is a limiting curve $\mathbf{z}$ of the family $\mathbf{z}^{\tau}$ by taking subsequence under the topology of uniform convergence on compact subsets as $\tau\downarrow0^+$.

Now, we shall compare two limiting curves, $\mathbf{x}$ obtained in the previous section and the maximizing movements $\mathbf{z}$. Let $\Delta$ be the partition of $[0,+\infty)$ with partition points $k\tau$, $k=0,1,\ldots$. Define a vector field $W_{\Delta}=W_\tau$ on $[0,\tau)\times M$ by
\begin{align*}
	W_{\tau}(t,x)=H_p(x,DT^+_{\tau-t}\phi(x)),\qquad t\in[0,\tau), x\in M,
\end{align*}
and extend $W_\tau$ to $[0,+\infty)\times M$ periodically, i.e., 
\begin{align*}
	W_\tau(t,x)=W_\tau(\langle t/\tau\rangle,\nabla T^+_{\tau-\langle t/\tau\rangle\tau}\phi(x))
\end{align*}
for $t\geqslant 0$ and $x\in M$. We denote by $\mathbf{x}^{\tau}_+$ the solution of the differential equation
\begin{equation}
	\begin{cases}
		\dot{\mathbf{x}}^\tau_+(t)=W_\tau(t,\mathbf{x}^\tau_+(t)),\quad t\geqslant0,\\
		\mathbf{x}^\tau_+(0)=x.
	\end{cases}
\end{equation}
Notice that $\mathbf{z}^\tau(t)=\mathbf{x}^\tau_+(t)$ at any partition point $t=k\tau$ with $k=0,1,\ldots$, and for any $t\in[0,+\infty)$
\begin{equation}\label{eq:x_z}
	\begin{split}
		|\mathbf{z}^\tau(t)-\mathbf{x}^\tau_+(t)|\leqslant&\,|\mathbf{z}^\tau(t)-\mathbf{z}^\tau(\lfloor t/\tau\rfloor\tau)|+|\mathbf{x}^\tau_+(\lfloor t/\tau\rfloor\tau)-\mathbf{x}^\tau(t)|\\
	\leqslant&\,\int^{\langle t/\tau\rangle \tau}_0|H_p(\mathbf{x}^\tau_+(t),\nabla T^+_{\tau-t}\phi(\mathbf{x}^{\tau}_+(t)))|\ dt\leqslant C\tau.
	\end{split}
\end{equation}

\begin{Pro}
The limiting curve $\mathbf{x}_+$ is a minimizing movements of $(\phi,H)$ from $x$ in positive direction. Each minimizing movements of $(\phi,H)$ from $x$ in positive direction is a generalized characteristic from $x$.
\end{Pro}

\begin{proof}
Let $\tau\to0^+$ in \eqref{eq:x_z} and we have $\mathbf{x}_+$ coincides with $\mathbf{z}$. The second statement in the theorem is a consequence of Proposition \ref{thm:Young}.
\end{proof}

We can also introduce the minimizing movements of $(\phi,H)$ from $x$ in negative direction similarly. Consider the minimizing problem: let $z^-_0=x$ and recursively define $z^-_{k+1}$ as a minimizer of the problem
\begin{align*}
	\inf\{\phi(y)+A_{\tau}(y,z^-_k)\},\quad k\in\N.
\end{align*}
For any $\tau\in(0,\tau(\phi)$, we define a pieacewise constant curve $\mathbf{x}_-^\tau:(-\infty,0]\to M$ by
\begin{align*}
	\mathbf{x}^{\tau}_-(t)=z^-_{\lfloor -t/\tau\rfloor},\quad t\leqslant0.
\end{align*}
We call any limiting curve $\mathbf{x}_-$ of $\{\mathbf{x}^{\tau}_-\}$, under the topology of convergence on compact subsets, a \emph{minimizing movement of $(\phi,H)$ from $x$ in negative direction}. By a similar argument as above, one can show that the limiting curve satisfies the differential inclusion \eqref{eq:GC} on $(-\infty,0]$. 

However, there is a essential difference between the minimizing movement of $(\phi,H)$ from $x$ in different diection, if $\phi$ is weak KAM solution of \eqref{eq:HJ_wk}. One can tell such difference under the following reasoning:
\begin{enumerate}
	\item For $z^-_0=x$, the function $\phi(\cdot)+A_{\tau}(\cdot,x)$ is semiconcave and its minimizer can be achieved at $z^-_1$. Note that such minimizer is not unique in general and there is a one-to-one correspondence between such minimizers and the element of $D^*\phi(x)$. 
	\item The function $\phi(\cdot)+A_{\tau}(\cdot,x)$ is differentiable at $z^-_1$ by the semiconcavity. Thus, $\phi$ is also differentiable at $z^-_1$ since the $A_\tau(\cdot,x)$ is of class $C^{1,1}_{\rm loc}$ when $t\in(0,\tau(\phi))$. By Fermat's rule,
	\begin{align*}
		D\phi(z^-_1)=-D_xA_\tau(x,z^-_1)=L_v(\xi^1(-\tau),\dot{\xi}^1(-\tau))
	\end{align*}
	where $\xi_1$ is the unique minimal curve for $A_\tau(x,z^-_1)$ defined on $[-\tau,0]$, since  the $A_\tau(\cdot,x)$ is of class $C^{1,1}_{\rm loc}$ again. We also have
	\begin{align*}
		L_v(\xi_1(0),\dot{\xi}_1(0))\in\nabla^*\phi(x).
	\end{align*}
	\item Since $\phi$ is differentiable at $z^-_1$, in the next step, there exists a unique minimizer $z^-_2$ of the function $\phi(\cdot)+A_{\tau}(\cdot,z^-_1)$. Let $\xi_2$ be the unique minimizer for $A_\tau(z^-_2,z^-_1)$ defined on $[-2\tau,-\tau]$. Then
	\begin{align*}
		D\phi(z^-_2)=&\,L_v(\xi_2(-2\tau),\dot{\xi}_2(-2\tau)),\\
		D\phi(z^-_1)=&\,L_v(\xi_2(-\tau),\dot{\xi}_2(-\tau))=L_v(\xi_1(-\tau),\dot{\xi}_1(-\tau)).
	\end{align*}
	Define $\gamma$ a curve defined on $[-2\tau,0]$ by
	\begin{align*}
		\gamma(t)=
		\begin{cases}
			\xi_2(t)& t\in[-2\tau,-\tau],\\
			\xi_1(t)& t\in[-\tau,0].
		\end{cases}
	\end{align*}
	It follows $\gamma$ is the minimal curve for $A_{2\tau}(z^-_2,x)$ such that $z^-_2$ is a minimizer of the problem
	\begin{align*}
		\inf\{\phi(y)+A_{2\tau}(y,x)\}.
	\end{align*}
	\item Inductively, there exists a $C^1$ curve $\xi:(-\infty,0]\to M$ such that
	\begin{align*}
		T^-_t\phi(x)=\phi(\gamma(t))+A_{-t}(\gamma(t),x),\quad\forall t<0,
	\end{align*}
	and $\gamma(-k\tau)=z^-_k$. Moreover, such $\gamma$ does not depends on the time step $\tau$, and $\gamma$ is exactly a minimizing movement of $(\phi,H)$ from $x$ by the same reasoning for the maximizing movement above. Moreover, each minimizing movement $\mathbf{x}_-$ of $(\phi,H)$ from $x$ satisfies the classical characteristic equation
	\begin{equation}\label{eq:classic_char}
		\dot{\mathbf{x}}_-(t)=H_p(\mathbf{x}_-(t),D\phi(\mathbf{x}_-(t))),\quad t\in(-\infty,0].
	\end{equation}
\end{enumerate}

In summary, we have

\begin{Pro}\label{pro:backward}
Suppose $\phi$ is a weak KAM solution of \eqref{eq:HJ_wk}. Then any minimizing movement of $(\phi,H)$ from $x$ in negative direction satisfies classical characteristic equation \eqref{eq:classic_char}. Moreover, there is a one-to-one correspondence of minimizing movements of $(\phi,H)$ from $x$ in negative direction and $D^*\phi(x)$.
\end{Pro}

Finally, we turn to the discussion of the relation between the generalized characteristics and the theory of \emph{homogenization}. For any $\tau\in(0,\tau_3(\phi)$, let $W:[0,\delta)\times M\to\R^n$ be a vector field defined by $W(t,x)=H_p(x,\nabla T^+_{\tau-t}\phi(x))$ and extend $W$ periodically such that $W(t+\tau,x)=W(t,x)$ for all $t\geqslant0$ and $x\in M$. We consider the following scaled system
\begin{align*}
	\begin{cases}
		\dot{\mathbf{x}}_k=W(kt,\mathbf{x}_k),\qquad t\geqslant0,\\
		\mathbf{x}_k(0)=x,
	\end{cases}
\end{align*}
and try to understand the limiting behavior of the solutions $\mathbf{x}_k$. Because of the uniform boundedness of the vector fields for all $k\in\N$, we can suppose $\mathbf{x}_k$ converges to a Lipschitz curve $\mathbf{x}$ uniformly on compact subsets as $m\to\infty$. This is a standard problem of homogenization of ordinary differential equations.

Since our construction of the vector field $W$ and the discussion above, the limiting curve $\mathbf{x}$ is a minimizing movement of $(\phi,H)$ from $x$ in positive direction and it is a generalized characteristic from $x$.

\subsection{Propagation of cut points along generalized characteristics}
\label{subsec:prop_cut}

We need a local $C^{1,1}$ property for the complement of cut locus, which is the key observation. For any weak KAM solution $\phi$ of \eqref{eq:HJ_wk} and $t>0$, let
\begin{align*}
	E(\phi,t):=\{x\in M: \tau_\phi(x)\geqslant t\},
\end{align*}
where $\tau_\phi$ is the cut time function for $\phi$. Notice that $\phi$ is differentiable at each point in $E(\phi,t)$.

\begin{Lem}[\cite{CCHW2024}]\label{lem:C11}
Suppose $H:T^*M\to\R$ is a Tonelli Hamiltonian. Then there exists positive constants $\lambda=
	\lambda_H,T=T_H,C=C_H$ such that for any weak KAM solution $\phi$ of \eqref{eq:HJ_wk}, $t\in(0,T]$ and $x\in E(\phi,t)$
	\begin{align*}
		|p-D\phi(x)|\leqslant\frac Ct|y-x|,\qquad\forall y\in B(x,\lambda t), p\in D^+\phi(y).
	\end{align*}
\end{Lem}

Using Lemma \ref{lem:C11} we can prove the following global propagation result.

\begin{The}[\cite{CCHW2024}]\label{thm:propagation_cut}
Suppose $H:T^*M\to\R$ is a Tonelli Hamiltonian, and $\phi$ is a weak KAM solution of \eqref{eq:HJ_wk}. Then the following hold
\begin{enumerate}
	\item For any $x\in M\setminus\text{\rm Cut}\,(\phi)$, let $\gamma_x:(-\infty,\tau_\phi(x)]\to M$ be the unique $(\phi,H)$-calibrated curve with $\gamma_x(0)=x$. Then $\gamma_x$ is the unique solution of the differential inclusion
	\begin{equation}\label{eq:GCI}\tag{GC}
		\begin{cases}
			\dot{\gamma}(t)\in\text{\rm co}\,H_p(\gamma(t),D^+\phi(\gamma(t))),\qquad a.e.\  t\in\R,\\
			\gamma(0)=x
		\end{cases}
	\end{equation}
	on $(-\infty,\tau_\phi(x)]$.
	\item If $\gamma:[0,+\infty)\to M$ satisfies \eqref{eq:GCI} and $\gamma(0)\in \text{\rm Cut}\,(\phi)$, then $\gamma(t)\in\text{\rm Cut}\,(\phi)$ for all $t\geqslant0$.
\end{enumerate}
\end{The}

\section{Generalized Hamiltonian gradient flow}
\label{sec:HGF}

\subsection{An EDI formalism of GHGF}

The notion of maximal slope curve can be understood as the steepest descent curve of a function, which plays an important role in the theory of gradient flows in metric space. Our treatment of maximal slope curves has the same spirit as the classical one. The readers can refer to the monograph \cite{Ambrosio_GigliNicola_Savare_book2008,Ambrosio_Brue_Semola_book2021} and the references therein for more details. 

Let $\phi$ be a semiconcave function on $M$, let $H$ be a Tonelli Hamiltonian and let $\mathbf{p}(x)$ be a Borel measurable selection of the superdifferential $D^+\phi(x)$. Let $\gamma:[0,t]\to M$ be an absolutely continuous curve. Then, for almost all $s\in[0,t]$ we have
\begin{align*}
	\min_{p\in D^+\phi(\gamma(s))}\langle p,\dot{\gamma}(s)\rangle&=\frac {d^+}{ds}\phi(\gamma(s))=\frac {d^-}{ds}\phi(\gamma(s))\\
		&=-\min_{p\in D^+\phi(\gamma(s))}\langle p,-\dot{\gamma}(s)\rangle=\max_{p\in D^+\phi(\gamma(s))}\langle p,\dot{\gamma}(s)\rangle.
\end{align*}
It follows that
\begin{align*}
	\frac {d}{ds}\phi(\gamma(s))=\langle p,\dot{\gamma}(s)\rangle,\qquad a.e. s\in[0,t],\ \forall p\in D^+\phi(\gamma(s)).
\end{align*}
Invoking Fenchel-Young inequality, we have that
\begin{align*}
	\phi(\gamma(t))-\phi(\gamma(0))&=\int_{0}^{t}\frac{d}{ds}\phi(\gamma(s))\ ds=\int_{0}^{t}\langle\mathbf{p}(\gamma(s)),\dot{\gamma}(s)\rangle\ ds\\
		&\leqslant \int_{0}^{t}\Big\{L(\gamma(s),\dot{\gamma}(s))+H(\gamma(s),\mathbf{p}(\gamma(s)))\Big\}\ ds,
\end{align*}
where equality holds if and only if
\begin{align*}
	\dot{\gamma}(s)=H_{p}(\gamma(s),\mathbf{p}(\gamma(s))),\quad a.e.\ s\in[0,t].
\end{align*}

Now, we introduce the notion of maximal slope curve and strict singular characteristic for a pair $(\phi,H)$. This approach is closely connected to the so-called EDI formalism in the theory of gradient flow.

\begin{defn}\label{defn:msc}
Let $\phi$ be a semiconcave function on $M$, let $H$ be a Tonelli Hamiltonian, and let $\mathbf{p}$ be a Borel measurable selection of the superdifferential $D^+\phi$.
\begin{enumerate}
	\item We call a locally absolutely continuous curve $\gamma:I\to M$ a \emph{maximal slope curve} for the pair $(\phi,H)$ and the selection $\mathbf{p}$, where $I$ is any interval which can be the whole real line, if $\gamma$ satisfies
	\begin{equation}\label{eq:msc_H_phi2}\tag{VI}
			\phi(\gamma(t_2))-\phi(\gamma(t_1))=\int_{t_1}^{t_2}\Big\{L(\gamma(s),\dot{\gamma}(s))+H(\gamma(s),\mathbf{p}(\gamma(s)))\Big\} ds
	\end{equation}
	for all $t_1,t_2\in I,\ t_1<t_2$, or, equivalently,
	\begin{equation}
		\dot{\gamma}(t)=H_{p}(\gamma(t),\mathbf{p}(\gamma(t))),\quad a.e.\ t\in I.
	\end{equation}
	\item For the \emph{minimal energy selection} $\mathbf{p}^{\#}_{\phi,H}$ define by
	\begin{equation}\label{eq:P^}
		\mathbf{p}^{\#}_{\phi,H}(x)=\arg\min\{H(x,p): p\in D^{+}\phi(x)\},\quad x\in M.
	\end{equation}
	we call any associated maximal slope curve $\gamma:I\to M$ for $(\phi,H)$ a \emph{strict singular characteristic} for the pair $(\phi,H)$. We use the term \emph{singular} in the definition since it is essentially connected to the phenomenon of propagation of singularities of $\phi$ in forward direction when the initial point is a singular point of $\phi$.
\end{enumerate}
\end{defn}

\begin{Rem}
If $\gamma:I\to M$ is a maximal slope curve for $(\phi,H)$ and $\mathbf{p}$, we have
\begin{align*}
	&\phi(\gamma(t_2))-\phi(\gamma(t_1))=\,\int^{t_2}_{t_1}\Big\{L(\gamma(s),\dot{\gamma}(s))+H(\gamma(s),\mathbf{p}(\gamma(s))\Big\}\ ds\\
	\geqslant&\,\int^{t_2}_{t_1}\Big\{L(\gamma(s),\dot{\gamma}(s))+H(\gamma(s),\mathbf{p}^\#_{\phi,H}(\gamma(s))\Big\}\ ds,\quad \forall t_1,t_2\in I,\ t_1<t_2.
\end{align*}
Since the converse inequality is true for any selection of $D^+\phi$, it follows that $\gamma$ is a strict singular characteristic for the pair $(\phi,H)$. This implies that \emph{the only meaningful maximal slope curve is exactly the strict singular characteristic determined by the minimal energy selection $\mathbf{p}^\#_{\phi,H}(x)$}. It is clear that the minimal energy selection $\mathbf{p}^\#_{\phi,H}(x)$ plays a crucial role in this theory, and the notions of maximal slope curve and strict singular characteristic coincide.
\end{Rem}

\subsection{Well-posedness of strict singular characteristics}
\label{subsec:wellposeness}

Thanks for the variational construction for strict singular characteristics, one has the following stability result.

\begin{The}[Stability of strict singular characteristics, \cite{CCHW2024}]\label{thm:stability}
	Let $\{H_k\}$ be a sequence of Tonelli Hamiltonians, $\{\phi_k\}$ be a sequence of $K$-semiconcave functions on $M$, and $\gamma_k:\R\to M$, $k\in\N$ be a sequence of strict singular characteristics for the pair $(\phi_k,H_k)$. We suppose the following condition:
	\begin{enumerate}
		\item $\phi_k$ converges to $\phi$ uniformly on $M$,
		\item $H_k$ converges to a Tonelli Hamiltonian $H$ uniformly on compact subset,
		\item $\gamma_k$ converges to $\gamma:\R\to M$ uniformly on compact subset.
	\end{enumerate}
	Then $\gamma$ is a strict singular characteristic for the pair $(\phi,H)$.
	
	Moreover, there exists a subsequence of strict singular characteristics $\{\gamma_{k_i}\}$ such that
	\begin{equation}\label{eq:P-converge}
		\lim_{i\to\infty}\mathbf{p}^\#_{\phi_{k_i},H_{k_i}}(\gamma_{k_i}(t))=\mathbf{p}^\#_{\phi,H}(\gamma(t)),\qquad a.e.\ t\in \R.
	\end{equation}
\end{The}

The stability in Theorem \ref{thm:stability} is also essential for the existence issue together with an approximation method (see, for instance, \cite{Yu2006,Cannarsa_Yu2009}). Let us recall the approximation method used in the paper \cite{Cannarsa_Yu2009}. It is known that for any semiconcave function $\phi$ on $M$ with Lipschitz constant $L_0$, semiconcavity constant $C_0$ and any fixed $x\in M$, there exists a sequence of smooth functions $\{\phi_k\}\subset C^{\infty}(M)$ such that:
\begin{enumerate}
	\item Each $\|D\phi_k\|_{C^0}\leqslant L_0$ and $D^2\phi_k$ is bounded above by $C_0I$ uniformly.
	\item $\phi_k$ converges to $\phi$ uniformly on $M$ as $k\to\infty$.
	\item $\lim_{k\to\infty}D\phi_k(x)=\mathbf{p}^{\#}_{\phi,H}(x)$ for given fixed $x\in M$. 
\end{enumerate}
Consider the differential equation
\begin{equation}\label{eq:SSC_k}
	\begin{cases}
		\dot{\gamma}(t)=H_p(\gamma(t),D\phi_k(\gamma(t))),\qquad t\in[0,\infty),\\
		\gamma(0)=x,
	\end{cases}
\end{equation}
and denote by $\gamma_k$ the unique solution of \eqref{eq:SSC_k}. The family $\{\gamma_k\}$ is equi-Lipschitz since $D\phi_k$ is uniformly bounded. Invoking the Ascoli-Arzel\`a theorem, by taking a subsequence if necessary, we can suppose that $\gamma_k$ converges to a Lipschitz curve $\gamma:[0,\infty)\to M$ uniformly on all compact subsets. In \cite{Cannarsa_Yu2009}, it is proved that there exists such a limiting curve $\gamma$ which is a generalized characteristic and it propagates singularities locally. Moreover, such a generalized characteristic satisfies the property that $\dot{\gamma}^+(0)$ exists and
\begin{equation}\label{eq:right_cont}
	\lim_{t\to0^+}\operatorname*{ess\ sup}_{s\in[0,t]}|\dot{\gamma}(s)-\dot{\gamma}^+(0)|=0.
\end{equation}
From Theorem \ref{thm:stability} we conclude that $\gamma$ is indeed a strict singular characteristic which means there is no convex hull in the right side of the generalized characteristic differential inclusion \eqref{eq:GC}.

\begin{The}[Existence of strict singular characteristics, \cite{CCHW2024}]\label{thm:msc exs}
	Let $\phi$ be a semiconcave function on $M$, $H$ be a Tonelli Hamiltonian. Then for any $x\in M$, there exists a strict singular characteristic $\gamma:\R\to M$ with $\gamma(0)=x$ for the pair $(\phi,H)$. In other words, 
	\begin{equation}\tag{SC}
		\begin{cases}
			\dot{\gamma}(t)=H_p(\gamma(t),\mathbf{p}^{\#}_{\phi,H}(\gamma(t))),\qquad a.e. \ t\in\R,\\
			\gamma(0)=x,
		\end{cases}
	\end{equation}
admits a Lipschitz solution.
\end{The}

Now, we ready to answer the question: \emph{what is the relation between the notion of strict singular characteristic and broken characteristic in \eqref{eq:SGC2}}? The answer is they two coincide!

\begin{The}[\cite{CCHW2024}]\label{thm:right_derivative}
Suppose $\phi$ is a semiconcave function on $M$, $H$ is a Tonelli Hamiltonian, and $\gamma:[0,T]\to M$ is a strict singular characteristic for the pair $(\phi,H)$.
\begin{enumerate}
	\item The right derivative $\dot{\gamma}^+(t)$ exists for all $t\in[0,T)$ and satisfies \eqref{eq:SGC2}.
	\item For all $t\in[0,T)$
	\begin{align*}
		\lim_{\tau\to t^+}\frac 1{\tau-t}\int^\tau_tH(\gamma(s),\mathbf{p}^\#_{\phi,H}(\gamma(s)))\ ds=H(\gamma(t),\mathbf{p}^\#_{\phi,H}(\gamma(t))).
	\end{align*}
\end{enumerate}
\end{The}

\subsection{More on intrinsic approach}
\label{subsec:intrinsic}

In subsection \ref{subsec:weak_conv}, by taking weak limit of the minimizing movement constructed from iterative intrinsic singular characteristic as the time step tends to $0$, one obtains a solution \eqref{eq:GC}. This relaxation approach determines the convex hull in \eqref{eq:GC} and this is parentally a obstacle for the uniqueness of \eqref{eq:GC}. However, the EDI formalism of maximal slope curve introduced in Definition \ref{defn:msc} allow us to construct strict singular characteristic from this intrinsic setting.

\begin{The}[\cite{CCHW2024}]\label{thm:solution limit}
Under the setting in subsection \ref{subsec:weak_conv}, fix $T>0$ and suppose $\{\Delta_k\}$ is a sequence of partitions of $[0,T]$ with $|\Delta_k|\leqslant\tau(\phi)$ for all $k\in\N$, and each $\gamma_k:[0,T]\to M$, $k\in\N$ is a solution of the equation
	\begin{equation}\label{eq:ODE}
		\dot{\gamma}_k(t)=W_{\Delta_k}(t,\gamma_k(t)),\qquad a.e.\ t\in[0,T].
	\end{equation}
Further assume that:
\begin{enumerate}
	\item $\lim_{k\to\infty}|\Delta_k|=0$,
	\item $\gamma_k$ converges uniformly to $\gamma:[0,T]\to M$.
\end{enumerate}
Then $\gamma$ is a strict singular characteristic for the pair $(\phi,H)$, that is,
\begin{equation}
	\dot{\gamma}(t)=H_p(\gamma(t),\mathbf{p}^{\#}_{\phi,H}(\gamma(t))),\qquad a.e.\ t\in[0,T].
\end{equation}
\end{The}

For the sake of further analysis, we need to sketch the proof. Let $\{\tau_i\}$ be the partition points of $[0,T]$, for every $i=0,\ldots,N-1$, 
\begin{align*}
	T^+_{\tau_{i+1}-\tau_{i}}\phi(\gamma(\tau_{i}))=\phi(\gamma(\tau_{i+1}))-\int^{\tau_{i+1}}_{\tau_{i}}L(\gamma,\dot{\gamma})\ ds.
\end{align*}
Thus
\begin{align*}
	T^+_{\tau_{i+1}-\tau_i}\phi(\gamma(\tau_{i}))-\phi(\gamma(\tau_{i}))&=\,\int^{\tau_{i+1}}_{\tau_i}\frac d{ds}T^+_{s-\tau_i}\phi(\gamma(\tau_{i}))\ ds\\
			&=\,\int^{\tau_{i+1}}_{\tau_i}H(\gamma(\tau_{i}),DT^+_{s-\tau_i}\phi(\gamma(\tau_{i})))\ ds.
\end{align*}
It follows that
\begin{align*}
	&\,\phi(\gamma(\tau_{i+1}))-\phi(\gamma(\tau_{i}))\\
			=&\,(\phi(\gamma(\tau_{i+1}))-T^+_{\tau_{i+1}-\tau_i}\phi(\gamma(\tau_{i})))+(T^+_{\tau_{i+1}-\tau_i}\phi(\gamma(\tau_{i}))-\phi(\gamma(\tau_{i})))\\
			=&\,\int^{\tau_{i+1}}_{\tau_i}L(\gamma(s),\dot{\gamma}(s))\ ds+\int^{\tau_{i+1}}_{\tau_i}H(\gamma(\tau_{i}),DT^+_{s-\tau_i}\phi(\gamma(\tau_{i})))\ ds.
\end{align*}
Summing up for $i=0,\ldots,N-1$ we obtain
\begin{align*}
	&\,\phi(\gamma(T))-\phi(\gamma(0))\\
	=&\,\int^T_{0}L(\gamma(s),\dot{\gamma}(s))\ ds+\sum^{N-1}_{i=0}\int^{\tau_{i+1}}_{\tau_i}H(\gamma(\tau_{i}),DT^+_{s-\tau_i}\phi(\gamma(\tau_{i})))\ ds\\
			=&\,\int^T_{0}\Big\{L(\gamma(s),\dot{\gamma}(s))+H(\gamma(\tau_{\Delta}^{-}(s)),DT^{+}_{s-\tau^-_{\Delta}(s)}(\gamma(\tau^-_{\Delta}(s))))\Big\}\ ds,
\end{align*}
where $\tau_{\Delta}^{-}(t)=\sup\{\tau_i|\,\tau_i\leqslant t\}$. The statements in Theorem \ref{thm:solution limit} then is proved by taking $|\Delta|\to0^+$. 

If $\phi$ is a weak KAM solution of \eqref{eq:HJ_wk}, $T^+_t\phi(x)-\phi(x)=0$ if $\tau_\phi(x)>0$ and $t\in[0,\tau_\phi(x))$. Thus, a key point in this intrinsic construction relies on how to understand the error term $T^+_t\phi(x)-\phi(x)$, especially when $x\in\text{Cut}\,(\phi)$. A very useful first order property 
\begin{align*}
	\frac d{dt}T^+_t\phi(x)=H(x,DT^+_t\phi(x)),\qquad \forall x\in M, t\ll1,
\end{align*}
validates this construction, especially the derivation of the energy dissipation term $H(x,\mathbf{p}^\#_{\phi,H}(x))$. 

We know that a semiconcave function $\phi$ satisfies the property
\begin{align*}
	(T^-_t\circ T^+_t)\phi(x)=\phi(x),\qquad\forall t\geqslant0, x\in M
\end{align*}
if and only if $\phi$ is a weak KAM solution (\cite{Cannarsa_Cheng_Hong2025}). Thus, the cut time function can be reformulated as 
\begin{align*}
	\tau_{\phi}(x)=\sup\{t\geqslant0: (T^-_t\circ T^+_t-T^+_t\circ T^-_t)\phi(x)=0\}.
\end{align*}
Recall $\text{Cut}\,(\phi)=\{x\in M:\tau_{\phi}(x)=0\}$. Then, the energy dissipation term $H(x,\mathbf{p}^\#_{\phi,H}(x))$ also embodies the non-commutativity of Lax-Oleinik operators $T^\pm_t$, and the creation and evolution of singularities. 

\subsection{Generalized Hamiltonian gradient flow}
\label{subsec:GHGF}

We have already discussed the existence and stability properties of \eqref{eq:SGC2} in section \ref{subsec:wellposeness}. However, it is still open that if this generazlied Hamiltonian gradient system admits a unique solution from any initial point, except for the case when $H(x,p)$ has Riemannian structure, i.e., $H$ is quadratic in $p$-variable. Thus, we can only regard  \eqref{eq:SGC2} as a differential inclusion. 

Now we adopt some idea from the theory of topological dynamical systems (see, for instance, \cite{Sell_book1971}). Let $M$ be a closed manifold and let $\pi:TM\to M$, $\pi(x,v)=x$ be the projection and $C^0_+:=C^0([0,+\infty),M)$. We define a push-forward semigroup $\mathbf{P}^t:C^0_+\to C^0_+$, $t\geqslant0$, by
\begin{align*}
	\mathbf{P}^t(\gamma)(s)=\gamma(t+s),\qquad\forall\gamma\in C^0_+,
\end{align*}
and the evaluation map $\mathbf{V}^t:C^0_+\to M$, $t\geqslant0$, by
\begin{align*}
	\mathbf{V}^t(\gamma)=\gamma(t).
\end{align*}
We introduce a distance
\begin{align*}
	d_+(\gamma_1,\gamma_2)=\sup_{t\geqslant0}e^{-t}d(\gamma_1(t),\gamma_2(t)),\qquad\gamma_1,\gamma_2\in C^0_+.
\end{align*}

\begin{Lem}
\hfill
\begin{enumerate}
	\item $(C^0_+,d_+)$ is a complete metric space. The topology induced by the distance $d_+$ is exactly the topology of convergence on compact subsets.
	\item The map $(t,\gamma)\mapsto\mathbf{P}^t(\gamma)$ is continuous on $[0,+\infty)\times C^0_+$.
	\item The map $(t,\gamma)\mapsto\mathbf{V}^t(\gamma)$ is continuous on $[0,+\infty)\times C^0_+$.
\end{enumerate}	
\end{Lem}

We regard $\mathcal{S}^+_{\phi,H}$, the set of all the solution of \eqref{eq:SGC2}, as a subspace of $(C^0_+,d_+)$. It is clear that $\mathbf{P}^t(\mathcal{S}^+_{\phi,H})\subset\mathcal{S}^+_{\phi,H}$ for all $t\geqslant0$. Thus, we lift the forward dynamics of \eqref{eq:SGC2} to a real semi-flow $(\mathcal{S}^+_{\phi,H},\mathbf{P}^t)$. 

\subsection{Mather measure as a maximal measure}
\label{subsec:Mather_max}

Now, we will study the invariant measures of the dynamical system $(\mathcal{S}^+_{\phi,H},\mathbf{P}^t)$. In the literature, this is closely connected to the \emph{occupational invariant measures} for differential inclusions (see \cite{Artstein1999}). We denote by $\mathbb{P}(\mathcal{S}^+_{\phi,H})$ the space of all Borel probability measures on $\mathcal{S}^+_{\phi,H}$, and $\mathcal{IM}(\phi,H)$ the set of all invariant measure of the semigroup $(\mathcal{S}^+_{\phi,H},\mathbf{P}^t)$. More precisely 
\begin{align*}
	\mathcal{IM}(\phi,H)=\{\mu\in \mathbb{P}(\mathcal{S}^+_{\phi,H}):(\mathbf{P}^t)_\#\mu=\mu, \forall t\geqslant0\}.
\end{align*}
Recall that
\begin{itemize}
	\item let $\widehat{\mathscr{M}_H}$ be the set of Mather measures on $TM$ of $H$;
	\item let $\mathscr{M}_H=\pi_{\#}(\widehat{\mathscr{M}_H})$ be the set of projected Mather measures on $M$ of $H$;
	\item let $\widehat{\mathscr{M}(H)}$ be the Mather set on $TM$ of $H$;
	\item let $\mathscr{M}(H)=\pi(\widehat{\mathscr{M}(H)})$ be the projected Mather set on $M$ of $H$.
	\item let $\{\Phi_L^t\}$ be the Euler-Lagrangian flow of $L$.
\end{itemize}

\begin{Pro}\label{pro:ineq}
Let $\phi\in\mathrm{SCL}\,(M)$ and $H$ be a Tonelli Hamiltonian.
\begin{enumerate}
	\item $(\mathcal{S}^+_{\phi,H},d_+)$ is a compact subspace of $(C^0_+,d_+)$.
	\item For each $\gamma\in\mathcal{S}^+_{\phi,H}$,
	\begin{align*}
		\limsup_{t\to+\infty}\frac 1t\int^t_0H(\gamma(s),\mathbf{p}_{\phi,H}^\#(\gamma(s))\ ds=-\liminf_{t\to+\infty}\frac 1t\int^t_0L(\gamma(s),\dot{\gamma}(s))\ ds\leqslant c[H].
	\end{align*}
	\item $\mathcal{IM}(\phi,H)$ is nonempty compact convex set in weak topology.
	\item For any $\mu\in\mathcal{IM}(\phi,H)$ we have
	\begin{align*}
		(\mathbf{V}^t)_\#\mu=(\mathbf{V}^0)_\#\mu,\qquad\forall t\geqslant0,
	\end{align*}
	and for any bounded Borel measurable function $f:M\to\R$
	\begin{align*}
		\int_{\mathcal{S}^+_{\phi,H}}f(\gamma(t))\ d\mu=\int_{\mathcal{S}^+_{\phi,H}}f(\gamma(0))\ d\mu,\qquad t\geqslant0.
	\end{align*}
	\item For any $\mu\in\mathcal{IM}(\phi,H)$ and $t\geqslant0$ we have
	\begin{align*}
		&\,\int_{\mathcal{S}^+_{\phi,H}}H(\gamma(t),\mathbf{p}_{\phi,H}^\#(\gamma(t)))\ d\mu=\int_{\mathcal{S}^+_{\phi,H}}H(\gamma(0),\mathbf{p}_{\phi,H}^\#(\gamma(0)))\ d\mu\\
		=&\,-\int_{\mathcal{S}^+_{\phi,H}}L(\gamma(0),\dot{\gamma}^+(0))\ d\mu=-\int_{\mathcal{S}^+_{\phi,H}}L(\gamma(t),\dot{\gamma}^+(t))\ d\mu\\
		\leqslant&\,c[H].
	\end{align*}
\end{enumerate}
\end{Pro}

\begin{proof}
The compactness of $(\mathcal{S}^+_{\phi,H},d_+)$ follows from the stability property of the strict singular characteristics (Theorem 3.7 in \cite{CCHW2024}) and Arzel\`a-Ascoli theorem.

To prove (2), we note that for any $\gamma\in\mathcal{S}^+_{\phi,H}$ we have that
\begin{align*}
	&\,\limsup_{t\to\infty}\frac 1t\int^t_0H(\gamma(s),\mathbf{p}_{\phi,H}^\#(\gamma(s))\ ds\\
	=&\,\lim_{t\to\infty}\frac 1t\int^t_0H(\gamma(s),\mathbf{p}_{\phi,H}^\#(\gamma(s))+L(\gamma(s),\dot{\gamma}(s))\ ds\\
	&\qquad\qquad +\limsup_{t\to\infty}\bigg\{-\int^t_0L(\gamma(s),\dot{\gamma}(s))\ ds\bigg\}\\
	=&\,\lim_{t\to\infty}\frac 1t[\phi(\gamma(t))-\phi(\gamma(0))]-\liminf_{t\to\infty}\int^t_0L(\gamma(s),\dot{\gamma}(s))\ ds\\
	=&\,-\liminf_{t\to\infty}\int^t_0L(\gamma(s),\dot{\gamma}(s))\ ds\leqslant-\liminf_{t\to\infty}\frac 1tA_t(\gamma(0),\gamma(t))\\
	=&\,-(-c[H])=c[H].
\end{align*}

(3) is a direct consequence of compactness of $(\mathcal{S}^+_{\phi,H},d_+)$ and the continuity of $\mathbf{P}^t$ for $t\geqslant0$.

For the proof of (4), we note that for any $t\geqslant0$ and $\mu\in\mathcal{IM}(\phi,H)$ we have 
\begin{align*}
	(\mathbf{V}^t)_\#\mu=(\mathbf{V}^0\circ\mathbf{P}^t)_\#\mu=(\mathbf{V}^0)_\#(\mathbf{P}^t)_\#\mu=(\mathbf{V}^0)_\#\mu
\end{align*}
and for any bounded Borel measurable function $f:M\to\R$ we have
\begin{align*}
	\int_{\mathcal{S}^+_{\phi,H}}f(\gamma(t))\ d\mu=\int_Mf\ d(\mathbf{V}^t)_\#\mu=\int_Mf\ d(\mathbf{V}^0)_\#\mu=\int_{\mathcal{S}^+_{\phi,H}}f(\gamma(0))\ d\mu.
\end{align*}

Now we turn to the proof of (5). In view of (4), for any $t\geqslant0$ and $\mu\in\mathcal{IM}(\phi,H)$ we have that
\begin{equation}\label{eq:inv1}
	\int_{\mathcal{S}^+_{\phi,H}}H(\gamma(t),\mathbf{p}_{\phi,H}^\#(\gamma(t)))\ d\mu=\int_{\mathcal{S}^+_{\phi,H}}H(\gamma(0),\mathbf{p}_{\phi,H}^\#(\gamma(0)))\ d\mu
\end{equation}
and
\begin{equation}\label{eq:inv2}
	\begin{split}
		&\,\int_{\mathcal{S}^+_{\phi,H}}L(\gamma(t),\dot{\gamma}^+(t))\ d\mu=\int_{\mathcal{S}^+_{\phi,H}}L(\gamma(t),H_p(\gamma(t),\mathbf{p}_{\phi,H}^\#(\gamma(t))))\ d\mu\\
		=&\,\int_{\mathcal{S}^+_{\phi,H}}L(\gamma(0),H_p(\gamma(t),\mathbf{p}_{\phi,H}^\#(\gamma(0))))\ d\mu=\int_{\mathcal{S}^+_{\phi,H}}L(\gamma(0),\dot{\gamma}^+(0))\ d\mu.
	\end{split}
\end{equation}
Combining \eqref{eq:inv1} and \eqref{eq:inv2} we conclude that
\begin{align*}
	&\,\int_{\mathcal{S}^+_{\phi,H}}H(\gamma(0),\mathbf{p}_{\phi,H}^\#(\gamma(0)))\ d\mu\\
	=&\,\frac 1t\int^t_0\int_{\mathcal{S}^+_{\phi,H}}H(\gamma(s),\mathbf{p}_{\phi,H}^\#(\gamma(s)))\ d\mu ds\qquad(\text{by}\ \eqref{eq:inv1})\\
	=&\,\frac 1t\int_{\mathcal{S}^+_{\phi,H}}\int^t_0H(\gamma(s),\mathbf{p}_{\phi,H}^\#(\gamma(s)))\ ds d\mu \qquad(\text{by Fubini's theorem})\\
	=&\,\frac 1t\int_{\mathcal{S}^+_{\phi,H}}\Big(\phi(\gamma(t))-\phi(\gamma(0))-\int^t_0L(\gamma(s),\dot{\gamma}^+(s))\ ds\Big)\ d\mu\\
	=&\,\frac 1t\int_{\mathcal{S}^+_{\phi,H}}\Big(\phi(\gamma(t))-\phi(\gamma(0))\Big)\ d\mu-\frac 1t\int^t_0\int_{\mathcal{S}^+_{\phi,H}}L(\gamma(s),\dot{\gamma}^+(s))\ d\mu ds\\
	&\qquad\qquad\qquad\qquad\qquad\qquad\qquad \qquad(\text{by Fubini's theorem})\\
	=&\,-\int_{\mathcal{S}^+_{\phi,H}}L(\gamma(0),\dot{\gamma}^+(0))\ d\mu,\qquad(\text{by (4) and \eqref{eq:inv2}})
\end{align*}
Finally, together with \eqref{eq:inv1} and the conclusion in (2) we obtain
\begin{align*}
	&\,\int_{\mathcal{S}^+_{\phi,H}}H(\gamma(0),\mathbf{p}_{\phi,H}^\#(\gamma(0)))\ d\mu\\
	=&\,\lim_{t\to\infty}\frac 1t\int^t_0\int_{\mathcal{S}^+_{\phi,H}}H(\gamma(s),\mathbf{p}_{\phi,H}^\#(\gamma(s)))\ d\mu ds\qquad(\text{by}\ \eqref{eq:inv1})\\
	=&\,\lim_{t\to\infty}\int_{\mathcal{S}^+_{\phi,H}}\bigg(\frac 1t\int^t_0H(\gamma(s),\mathbf{p}_{\phi,H}^\#(\gamma(s)))\ ds\bigg)\ d\mu \qquad(\text{by Fubini's theorem})\\
	\leqslant&\,\int_{\mathcal{S}^+_{\phi,H}}\limsup_{t\to\infty}\bigg(\frac 1t\int^t_0H(\gamma(s),\mathbf{p}_{\phi,H}^\#(\gamma(s)))\ ds\bigg)\ d\mu \qquad(\text{by Fatou's lemma})\\
	\leqslant&\,\int_{\mathcal{S}^+_{\phi,H}}c[H]\ d\mu=c[H],\qquad(\text{by (2))}.
\end{align*}
This completes the proof.
\end{proof}


Recalling the notation $\Gamma$ in Section \ref{subsec:intro3}, we have the following new characterization of Ma\~n\'e's critical value as well as projected Mather sets.

\begin{The}\label{thm:Mather_max}
\hfill
\begin{enumerate}
	\item We have the following characterization of Ma\~n\'e critical value.
	\begin{equation}\label{eq:Mane}
		c[H]=\max_{\phi\in\mathrm{SCL}\,(M)}\sup_{\mu\in\mathcal{IM}(\phi,H)}\int_{\mathcal{S}^+_{\phi,H}}H(\gamma(0),\mathbf{p}_{\phi,H}^\#(\gamma(0)))\ d\mu.
	\end{equation}
	\item If $\phi$ is a weak KAM solution of \eqref{eq:HJ_wk}, $\mu_0\in\mathscr{M}_H$, then $\mu=\Gamma_\#\mu_0\in\mathcal{IM}(\phi,H)$ and $(\mathbf{V}^0)_\#\mu=\mu_0$. Moreover, 
	\begin{equation}\label{eq:Mane1}
		\int_{\mathcal{S}^+_{\phi,H}}H(\gamma(0),\mathbf{p}_{\phi,H}^\#(\gamma(0)))\ d\mu=c[H].
	\end{equation}
	\item If $\phi\in\mathrm{SCL}\,(M)$ and $\mu\in\mathcal{IM}(\phi,H)$ satisfy \eqref{eq:Mane1}, then $\mu_0=(\mathbf{V}^0)_\#\mu\in\mathscr{M}_H$ and $\mu=\Gamma_\#\mu_0$.
\end{enumerate}	
\end{The}

\begin{proof}
Let $\mu_0\in\mathscr{M}_H$. Then $\mu=\Gamma_\#\mu_0\in\mathbb{P}(\mathcal{S}^+_{\phi,H})$ and
\begin{align*}
	(\mathbf{V}^0)_\#\mu=(\mathbf{V}^0)_\#\Gamma_\#\mu_0=(\mathbf{V}^0\circ\Gamma)_\#\mu_0=\text{id}_\#\mu_0=\mu_0.
\end{align*}
If $\phi$ is a weak KAM solution of \eqref{eq:HJ_wk}, then
\begin{align*}
	\int_{\mathcal{S}^+_{\phi,H}}H(\gamma(0),\mathbf{p}_{\phi,H}^\#(\gamma(0)))\ d\mu=&\,\int_MH(x,\mathbf{p}^\#_{\phi,H}(x))\ d(\mathbf{V}^0)_\#\mu\\
	=&\,\int_MH(x,D\phi(x))\ d\mu_0=c[H].
\end{align*}
Now we show that $\mu$ is invariant measure of the semi-flow $\{\mathbf{P}^t\}_{t\geqslant0}$. Observe that
\begin{align*}
	\mathbf{P}^t\circ\Gamma(x)=\Gamma(\gamma(x,t))=\Gamma\circ\pi\circ\Phi^t_L\circ f_L(x),\qquad\forall t\geqslant0, \forall x\in\mathscr{M}(H),
\end{align*}
and $\hat{\mu}=(f_L)_\#\mu_0\in\widehat{\mathscr{M}_H}$, $\pi_\#\hat{\mu}=\mu_0$. Therefore 
\begin{align*}
	(\mathbf{P}^t)_\#\mu=&\,(\mathbf{P}^t)_\#\Gamma_\#\mu_0=\Gamma_\#\pi_\#(\Phi^t_L)_\#(f_L)_\#\mu_0\\
	=&\,\Gamma_\#\pi_\#(\Phi^t_L)_\#\hat{\mu}=\Gamma_\#\pi_\#\hat{\mu}=\Gamma_\#\mu_0=\mu,\qquad\forall t\geqslant0.
\end{align*}
Thus $\mu\in\mathcal{IM}(\phi,H)$ and \eqref{eq:Mane} is a consequence together with Proposition \ref{pro:ineq} (v). This completes the proof of (2) and (1).

Given $\phi\in\mathrm{SCL}\,(M)$, $\mu\in\mathcal{IM}(\phi,H)$ satisfy \eqref{eq:Mane1} and $\mu_0=(\mathbf{V}^0)_\#\mu$. Define $g_\phi:M\to TM$, $g_\phi(x)=(x,H_p(x,\mathbf{p}^\#_{\phi,H}(x))$ and let $\tilde{\mu}=(g_\phi)_\#\mu_0=(g_\phi)_\#(\mathbf{V}^0)_\#\mu$. Observe that $\pi_\#\tilde{\mu}=\pi_\#(g_\phi)_\#\mu_0=\mu_0$. To show $\mu_0\in\mathscr{M}_H$ it suffices to show $\tilde{\mu}\in\widehat{\mathscr{M}_H}$.

In view of Proposition \ref{pro:ineq} (v), for any $t>0$,
\begin{align}
	&\,\frac 1t\int_{\mathcal{S}^+_{\phi,H}}\bigg(\int^t_0L(\gamma(s),\dot{\gamma}^+(s))\ ds\bigg)\ d\mu\notag\\
	=&\,\frac 1t\int^t_0\bigg(\int_{\mathcal{S}^+_{\phi,H}}L(\gamma(s),\dot{\gamma}^+(s))\ d\mu\bigg)\ ds\ (\text{by Fubini's theorem)}\notag\\
	=&\,\frac 1t\int^t_0\bigg(-\int_{\mathcal{S}^+_{\phi,H}}H(\gamma(0),\mathbf{p}_{\phi,H}^\#(\gamma(0)))\ d\mu\bigg)\ ds\\
	&\qquad\qquad\qquad\qquad\qquad\qquad\qquad\qquad (\text{by Proposition \ref{pro:ineq} (5))}\notag\\
	=&\,\frac 1t\int^t_0(-c[H])\ ds=-c[H].\label{eq:eq1}
\end{align}
Since $\mu\in\mathcal{IM}(\phi,H)$, then for any $k\in\N$,
\begin{align*}
	&\,\int_{\mathcal{S}^+_{\phi,H}}A_t(\gamma(0),\gamma(t))\ d\mu=\int_{\mathcal{S}^+_{\phi,H}}A_t(\mathbf{V}^0(\gamma),\mathbf{V}^t(\gamma))\ d(\mathbf{P}^{kt})_\#\mu\\
	=&\,\int_{\mathcal{S}^+_{\phi,H}}A_t(\mathbf{V}^0(\mathbf{P}^{kt}(\gamma)),\mathbf{V}^t(\mathbf{P}^{kt}(\gamma)))\ d\mu=\int_{\mathcal{S}^+_{\phi,H}}A_t(\gamma(kt),\gamma((k+1)t))\ d\mu.
\end{align*}
Therefore
\begin{align}
	&\,\frac 1t\int_{\mathcal{S}^+_{\phi,H}}A_t(\gamma(0),\gamma(t))\ d\mu\\
	=&\,\lim_{n\to\infty}\frac 1{nt}\sum^{n-1}_{k=0}\int_{\mathcal{S}^+_{\phi,H}}A_t(\gamma(kt),\gamma((k+1)t))\ d\mu\notag\\
	=&\,\lim_{n\to\infty}\int_{\mathcal{S}^+_{\phi,H}}\bigg(\frac 1{nt}\sum^{n-1}_{k=0}A_t(\gamma(kt),\gamma((k+1)t))\bigg)\ d\mu\\
	\geqslant&\,\liminf_{n\to\infty}\int_{\mathcal{S}^+_{\phi,H}}\frac 1{nt}A_t(\gamma(0),\gamma(nt))\ d\mu\notag\\
	=&\,\int_{\mathcal{S}^+_{\phi,H}}\bigg(\lim_{n\to\infty}\frac 1{nt}A_t(\gamma(0),\gamma(nt))\bigg)\ d\mu\notag\\
	=&\,\int_{\mathcal{S}^+_{\phi,H}}(-c[H])\ d\mu=-c[H].\label{eq:eq2}
\end{align}
The combination of \eqref{eq:eq1} and \eqref{eq:eq2} yields 
\begin{align*}
	\frac 1t\int_{\mathcal{S}^+_{\phi,H}}\bigg(\int^t_0L(\gamma(s),\dot{\gamma}^+(s))\ ds-A_t(\gamma(0),\gamma(t))\bigg)\ d\mu\leqslant0.
\end{align*}
Since $\int^t_0L(\gamma(s),\dot{\gamma}^+(s))\ ds-A_t(\gamma(0),\gamma(t))\geqslant0$ for any $\gamma\in\mathcal{S}^+_{\phi,H}$, we conclude that
\begin{align*}
	\int^t_0L(\gamma(s),\dot{\gamma}^+(s))\ ds=A_t(\gamma(0),\gamma(t)),\qquad\mu-a.e.\ \gamma\in\mathcal{S}^+_{\phi,H}.
\end{align*}
In other words, $\gamma\vert_{[0,t]}$ is a minimizer for $A_t(\gamma(0),\gamma(t))$. This implies that
\begin{equation}\label{eq:inv3}
	\Phi^t_L(\gamma(0),H_p(\gamma(0),\mathbf{p}^\#_{\phi,H}(\gamma(0)))=(\gamma(t),H_p(\gamma(t),\mathbf{p}^\#_{\phi,H}(\gamma(t))).
\end{equation}
i.e., $\Phi^t_L\circ g_\phi\circ\mathbf{V}^0(\gamma)=g_\phi\circ\mathbf{V}^t(\gamma)$. Together with Proposition \ref{pro:ineq} (4), it yields
\begin{align*}
	(\Phi^t_L)_\#\tilde{\mu}=(\Phi^t_L)_\#(g_\phi)_\#(\mathbf{V}^0)_\#\mu=(g_\phi)_\#(\mathbf{V}^t)_\#\mu=(g_\phi)_\#(\mathbf{V}^0)_\#\mu=\tilde{\mu}.
\end{align*}
Therefore $\tilde{\mu}$ is $\Phi^t_L$-invariant. Moreover, due to Proposition \ref{pro:ineq} (5),
\begin{align*}
	\int_{TM}L(x,v)\ d\tilde{\mu}=&\,\int_{TM}L(x,v)\ d(g_\phi)_\#(\mathbf{V}^0)_\#\mu\\
	=&\,\int_ML(x,H_p(x,\mathbf{p}^\#_{\phi,H}(x))\ d(\mathbf{V}^0)_\#\mu\\
	=&\,\int_{\mathcal{S}^+_{\phi,H}}L(\gamma(0),\dot{\gamma}^+(0))\ d\mu\\
	=&\,-\int_{\mathcal{S}^+_{\phi,H}}H(\gamma(0),\mathbf{p}^\#_{\phi,H}(\gamma(0)))\ d\mu=-c[H].
\end{align*}
It follows that $\tilde{\mu}\in\widehat{\mathscr{M}_H}$ and $\mu_0=\pi_\#\tilde{\mu}\in\mathscr{M}_H$. By \eqref{eq:inv3} we see that for $\mu$-a.e. $\gamma\in\mathcal{S}^+_{\phi,H}$ it holds $\gamma=\Gamma\circ\mathbf{V}^0(\gamma)$. Thus, $\mu=\Gamma_\#(\mathbf{V}^0)_\#\mu=\Gamma_\#\mu_0$.
\end{proof}

\begin{Rem}
\hfill
\begin{enumerate}
	\item For general $\phi\in\mathrm{SCL}\,(M)$, the measure $\mu$ satisfying item (3) in Theorem \ref{thm:Mather_max} may not exist.
	\item If there exists $\mu$ making Theorem \ref{thm:Mather_max} (3) hold true, it is unnecessary that $\phi$ is a weak KAM solution of \eqref{eq:HJ_wk} and such a function $\phi$ is not necessary to be differentiable on the support of $\mu_0$. See the following example. 
\end{enumerate}
\end{Rem}

\begin{Ex}
Consider the Hamiltonian $H:T^*\mathbb{S}\to\R$, $H(x,p)=\frac 12|p|^2-\cos x$. In this case $\widehat{\mathscr{M}_H}=\{\delta(\pi,0)\}$ with $\delta(\pi,0)$ the Dirac measure supported on $(\pi,0)$, and $\mathscr{M}_H=\{\delta(\pi)\}$ with $\delta(\pi)$ the Dirac measure supported on $\pi$. We also have $\widehat{\mathscr{M}(H)}=\{(\pi,0)\}$, $\mathscr{M}(H)=\{\pi\}$ and $c[H]=1$. Now, let
\begin{align*}
	\phi(x)=
	\begin{cases}
		-|x-\pi|,&x\in(\frac{\pi}2,\frac 32\pi),\\
		\frac 1{\pi}x^2-\frac 34\pi,&x\in[-\frac{\pi}2,\frac{\pi}2].
	\end{cases}
\end{align*}
Then $\phi\in\mathrm{SCL}\,(\mathbb{S})$ but $\phi$ is not a weak KAM solution of the associated Hamilton-Jacobi equation \eqref{eq:HJ_wk}.

Let $\gamma_\pi(t)=\pi$ for $t\in[0,+\infty)$, then $\gamma_\pi\in\mathcal{S}^+_{\phi,H}$ and $\mathbf{P}^t(\gamma_\pi)=\gamma_\pi$ for all $t\in[0,+\infty)$. Therefore $\delta(\gamma_\pi)$, Dirac measure supported on $\gamma_\pi$, is contained in $\mathcal{IM}(\phi,H)$, and
\begin{align*}
	\int_{\mathcal{S}^+_{\phi,H}}H(\gamma(0),\mathbf{p}^\#_{\phi,H}(\gamma(0)))\ d\delta(\gamma_\pi)=H(\pi,0)=1=c[H].
\end{align*}
Notice that $(\mathbf{V}^0)_\#(\delta(\gamma_\pi))=\delta(\pi)$, but $\phi$ is not differentiable at $x=\pi$.
\end{Ex}

\section{Viewpoint from Lagrangian to Eulerian}
\label{sec:Eulerian}

\subsection{Mass transport aspect of GHGF}
\label{subsec:mass_transport}

To study the mass transport along strict singular characteristics, one cannot apply the standard DiPerna-Lions theory to deduce the continuity equation for which the solution is a curve of probability measures determined by the measures driven by the flow of the vector field, since the collision and focus of the classical characteristics. However, one can use a lifting method  in section \ref{subsec:GHGF}. 

Let $\phi:M\to\R$ be a semiconcave function, $H$ be a Tonelli Hamiltonian, and $(\mathcal{S}^+_{\phi,H},\mathbf{P}^t)$ the semi-flow introduced in section \ref{subsec:GHGF}. Consider the set-valued map $\Gamma:M\rightrightarrows\mathcal{S}^+_{\phi,H}$ defined by
\begin{align*}
	x\mapsto\Gamma(x)=\{\gamma\in\mathcal{S}^+_{\phi,H}: \gamma(0)=x\}.
\end{align*}
The set-valued map $\Gamma$ is nonempty, compact-valued and measurable, and $\Gamma$ admits a measurable selection $\gamma:M\to\mathcal{S}^+_{\phi,H}$, $x\mapsto\gamma(x,\cdot)$	such that $\gamma(x,0)=x$ for all $x\in M$. For any such measurable map $\gamma$ we introduce a map
\begin{align*}
	\Phi_\gamma^t(x)=\gamma(x,t),\qquad t\in[0,T],\ x\in M.
\end{align*}
For any $t\in[0,T]$, $\Phi_{\gamma}^t:M\to M$ is measurable, and $\Phi^0_\gamma=\text{id}$.

\begin{The}[\cite{CCHW2024}]\label{thm:CE}
Suppose $\phi$ is a semiconcave function on $M$, $H:T^*M\to\R$ is a Tonelli Hamiltonian, and $\gamma:M\to\mathcal{S}^+_{\phi,H}$ is a a measurable selection of $\Gamma:M\rightrightarrows\mathcal{S}^+_{\phi,H}$. Then for any $\bar{\mu}\in\mathscr{P}(M)$, the set of Borel probability measures on $M$, the curve $\mu_t=(\Phi^t_\gamma)_{\#}\bar{\mu}$, $t\in[0,T]$ in $\mathscr{P}(M)$ satisfies the continuity equation
\begin{equation}\label{eq:CE}\tag{CE}
	\begin{cases}
		\frac d{dt}\mu+\text{\rm div}(H_p(x,\mathbf{p}^{\#}_{\phi,H}(x))\cdot\mu)=0,\\
		\mu_0=\bar{\mu}.
	\end{cases}
\end{equation}
and we have the following properties:
\begin{enumerate}
	\item $\int_M\phi\mu_T-\int_M\phi\mu_0=\int^T_0\int_M\big\{L(x,H_p(x,\mathbf{p}^\#_{\phi,H}(x)))+H(x,\mathbf{p}^\#_{\phi,H}(x))\big\}\ d\mu_sds$.
	\item For any $t\in[0,T)$, $g\in C^\infty(M)$,
	\begin{align*}
		\frac{d^+}{dt}\int_Mg\ d\mu_t=\int_M\nabla g(x)\cdot H_p(x,\mathbf{p}^\#_{\phi,H}(x))\ d\mu_t.
	\end{align*}
	\item For any $t\in[0,T)$ and $f\in C_c(T^*M,\R)$
	\begin{align*}
		\int_Mf(x,\mathbf{p}^\#_{\phi,H}(x))\ d\mu_t=\lim_{\tau\to t^+}\frac 1{\tau-t}\int^\tau_t\int_Mf(x,\mathbf{p}^\#_{\phi,H}(x))\ d\mu_sds.
	\end{align*}
\end{enumerate}
\end{The}

\begin{The}[\cite{CCHW2024}]\label{thm:mass sing}
Under the assumption of Theorem \ref{thm:CE}, if $H$ is a Tonelli Hamiltonian and $\phi$ is a weak KAM solution of \eqref{eq:HJ_wk}, we have that, for all $0\leqslant t_1\leqslant t_2\leqslant T$
\begin{align*}
	\mu_{t_1}(\text{\rm Cut}\,(\phi))\leqslant\mu_{t_2}(\text{\rm Cut}\,(\phi)),\quad \mu_{t_1}(\overline{\text{\rm Sing}\,(\phi)})\leqslant\mu_{t_2}(\overline{\text{\rm Sing}\,(\phi)}).
\end{align*}
\end{The}

\subsection{Dynamical cost functional and Random Lax-Oleinik semigroup}

In subsection \ref{subsec:mass_transport} we explained the mass transport nature of the GHGF. We will concentrate on their optimal transport nature and bridge these results and the problem of propagation of singularities in a specific mean field model.

To address singularities and cut locus of functionals on Wasserstein spaces, we follow the framework in \cite{Bonnet_Frankowska2022}. For a metric space \((X, d)\), we denote by \(\mathscr{P}(X)\) the set of all probability Borel measures on \(X\), and by \(\mathscr{P}_c(\mathbb{R}^m)\) the subspace of \(\mathscr{P}(\R^m)\) consisting of measures with compact support.

For any functional \(U: \mathscr{P}_c(\mathbb{R}^m) \to \mathbb{R}\), we introduce \(\partial^{\pm} U(\mu)\), the localized Fréchet superdifferential and subdifferential of \(U\) at \(\mu\). 

\begin{defn}[Localized Fr\'{e}chet generalized differential]\label{def:local frechet}
Assume that $U: \mathscr P_c(\R^m)\to\R$, $\mu\in\mathscr P_c(\R^m)$, we say $\alpha\in L^2(\R^m;\mu)$ belongs to $\partial^+U(\mu)=\partial^+_{\rm loc}U(\mu)$, the set of \textit{localized Fr\'{e}chet superdifferential} of $U$ at $\mu$, if for any $R>0$ and any $\nu\in\mathscr{P}(B_R(\mu))$, 
\begin{equation}\label{eq:wass-supdiff}
	U(\nu)-U(\mu) \leqslant \sup_{\gamma \in \Gamma_2(\mu, \nu)}\left\{\int_{\mathbb{R}^{2m}}\langle\alpha(x), y-x\rangle\,d\gamma\right\}+o_R(W_2(\mu,\nu)),
\end{equation}
where $o_R(W_2(\mu,\nu))$, depending on $R$, represents the high order infinitesimal of $W_2(\mu,\nu)$.

Similarly, $\beta\in L^2(\R^m;\mu)$ belongs to $\partial^-U(\mu)=\partial^-_{\rm loc}U(\mu)$, the set of \textit{localized Fr\'{e}chet subdifferential} of $U$ at $\mu$, if for any $R>0$ and any $\nu\in\mathscr{P}(B_R(\mu))$,
\begin{equation}\label{eq:wass-subdiff}
	U(\nu)-U(\mu) \geqslant \inf_{\gamma \in \Gamma_2(\mu, \nu)}\left\{\int_{\mathbb{R}^{2m}}\langle\beta(x), y-x\rangle\,d\gamma\right\}+o_R(W_2(\mu,\nu)).
\end{equation}

For a time dependent functional $U:\R\times\mathscr P_c(\R^m)\to\R$, we say  $(q,\alpha)\in\R\times L^2(\R^m;\mu)\in \partial^+U(t,\mu)$, if for any $R>0$ and $a,b\in\R$ satisfying $t\in(a,b)$ and $(s,\nu)\in[a,b]\times\mathscr P(B_R(\mu))$,
\begin{align*}
	&\,U(s,\nu)-U(t,\mu)\\
	\leqslant&\,\sup_{\gamma \in \Gamma_2(\mu, \nu)}\left\{\int_{\mathbb{R}^{2m}}\langle\alpha(x), y-x\rangle\,d\gamma\right\}+q(s-t)+o_R(W_2(\mu,\nu)+|t-s|).
\end{align*}The definition of $\partial^-U(t,\mu)$ is similar.
\end{defn}

\begin{defn}[Singularity of functional in $\mathscr P_c(\R^m)$]\label{def:Pc-diff-sing}
Let $U:\mathscr P_c(\R^m)\to\R$ and $\mu\in\mathscr P_c(\R^m)$. 
\begin{enumerate}
	\item $\mu$ is called a \emph{regular point} of $U$ if both $\partial^+ U(\mu)$ and $\partial^-U(\mu)$ are non-empty. We also say $U$ is \emph{locally differentiable} at $\mu$.
	\item $\mu$ is called a \emph{singular point} of $U$ if $U$ is not locally differentiable at $\mu$. We denote by $\mathscr{S}(U)$ the set of all singular points of $U$.
\end{enumerate}
\end{defn}

Now we consider the case that $L$ is a Tonelli Lagrangian and the time-dependent cost function $c^t(x,y)=A_t(x,y)$ is the fundamental solution with respect to $L$. In this context, a connection from optimal transportation to Mather theory and weak KAM theory is firstly studied by Bernard and Buffoni in \cite{Bernard_Buffoni2007a}. 

\begin{defn}[Dynamical cost functional]
The \emph{dynamical cost functional} of $c^{t}(\cdot,\cdot)$ is defined as
\begin{align*}
	C^{t}(\mu,\nu):=\inf_{\gamma\in\Gamma(\mu,\nu)}\int_{M\times M}c^{t}(x,y)\,d\gamma
\end{align*}
We denote by $\Gamma_o^{t}(\mu,\nu)$ the set of minimizers of $C^{t}(\mu,\nu)$. 
\end{defn}

\begin{defn}[Random curves]
For a given probability space $(\Omega,\mathscr F,\mathbb P)$ and $T>0$, An absolutely continuous \emph{random curve} on $[0,T]$ refers to a measurable map $\xi: [0,T]\times\Omega\to M$, which satisfies for any fixed $\omega\in\Omega$, sample path $\xi(\cdot,\omega)$ is absolutely continuous curve on $M$. Moreover, $\xi=\xi(s)$ can also be treated as a random process for $s\in[0,T]$.
\end{defn}

For the setting on work space $\mathscr P_c(\R^m)$, We select the probability measure space $(\Omega,\mathscr F,\mathbb P)$ for the random curves. 
In the subsequent discussion, we usually omit the component $\omega\in\Omega$ of the random curves and denote the expectation of a random variable (or vector) on 
$(\Omega,\mathscr F,\mathbb P)$ by $\mathbb E$ for simplicity when no confusion arises.

\begin{defn}[Dynamical coupling]
We say an absolutely continuous random curve $\xi$ is a \emph{dynamical coupling} of $\mu$ and $\nu$ for $t>0$, if $\text{\rm law}(\xi(0))=\mu$ and $\text{\rm law}(\xi(t))=\nu$. The set of all dynamical couplings of $\mu$ and $\nu$  for $t>0$ is denoted by $L_{\mu,\nu}^t$. 
\end{defn}

\begin{Pro}[\cite{Bernard_Buffoni2007a,Villani_book2009}]\label{pro:exist-displacement}
For any $\nu_1,\nu_2\in\mathscr P(\R^m)$ such that $C^{t}(\nu_1,\nu_2)$ is finite, 
\begin{align*}
	C^t(\nu_1,\nu_2)=\min_{\xi\in L_{\nu_1,\nu_2}^t}\E\left(\int_0^tL(\xi(s),\dot\xi(s))\,ds\right).
\end{align*}
The minimizer of the last term is called dynamical optimal coupling of $\nu_1$ and $\nu_2$ for $t>0$. $\mu_s:=\text{\rm law}(\xi_*(s))$ $(s\in[0,t])$ as the law of a dynamical optimal coupling $\xi_*$ is called a displacement interpolation of $C^t(\nu_1,\nu_2)$. In this case, we have
\begin{itemize}
	\item $\xi_*$ is a random solution of the Euler-Lagrange equation;
	\item the path $\{\mu_s\}_{s\in[0,t]}$ is a minimizing curve for the action functional defined on $\mathscr P(\R^m)$ by
	\begin{align*}
		\mathbb A^t(\{\mu_s\}_{s\in[0,t]}):=\min_{\xi}\E\left(\int_0^tL(\xi(s),\dot\xi(s))\,ds\right),
	\end{align*}where the last minimum is over all random curves $\xi$ such that $\text{\rm law}(\xi(s))=\mu_s$ for $s\in[0,t]$.
\end{itemize}
\end{Pro}
For more discussions on dynamical optimal coupling and displacement interpolation, readers can refer to Chapter 7 in \cite{Villani_book2009}.

\begin{defn}[Random Lax-Oleinik operators]\label{def:Pt-,Pt+}
For any Borel measurable function $\phi:\R^m\to [-\infty,+\infty]$, $t>0$ and $\mu\in\mathscr P(\R^m)$, we define
\begin{align}
	P_t^+\phi(\mu)&:=\sup_{\nu\in\mathscr P(\R^m)}\{\phi(\nu)-C^t(\mu,\nu)\},\label{eq:def-pt+}\\
	P_t^-\phi(\mu)&:=\inf_{\nu\in\mathscr P(\R^m)}\{\phi(\nu)+C^t(\nu,\mu)\}.\label{eq:def-pt-}
\end{align}
$P_t^+$ and $P_t^-$ are called \emph{positive} and \emph{negative type random Lax-Oleinik operators} respectively.
\end{defn}

The following theorem and corollary show that $\mathscr P_c(\R^m)$ is a suitable working space for the operator $P_t^{\pm}$. 

\begin{Pro}[\cite{CCSW2025}]\label{pro:probability lax-oleinik}
Let $\phi\in \UC(\R^m)$, $t_0>0$ and $c[0]=0$. Random Lax-Oleinik operators have the following properties:
\begin{enumerate}
		\item $\{P_t^-\}_{t>0}$ and $\{P_t^+\}_{t>0}$ are semigroups with respect to $t\in(0,+\infty)$;
		\item $\lim_{t\to 0^+}P_t^-\phi=\lim_{t\to 0^+}P_t^+\phi=\phi$, thus we can define $P_0^-\phi=P_0^+\phi=\phi$;
		\item for arbitrary $t,s\geqslant 0$, $P_t^-\circ P_s^+\phi=T_t^-\circ T_s^+\phi(\cdot)$; $P_t^+\circ P_s^-\phi=T_t^-\circ T_s^-\phi(\cdot)$;
		\item for any $t>0$, $P_t^-\phi$ (resp. $-P_t^+\phi$) is locally (resp. strongly) semiconcave and $\{P_t^-\phi\}_{t\geqslant t_0}$ (resp. $\{-P_t^+\phi\}_{t\geqslant t_0}$) uniformly localized (resp. strongly) semiconcave;
		\item $t\mapsto P_t^-\phi$ ($t\mapsto P_t^+\phi$) uniformly continuous on $[0,+\infty)$; 
		\item suppose $K\subset\R^m$ is compact set. $(t,\mu)\mapsto P_t^-\phi(\mu)$ ($(t,x)\mapsto P_t^+\phi(x)$) is continuous on $[0,+\infty)\times \mathscr P(K)$, locally Lipschitz on $(0,+\infty)\times \mathscr P(K)$ and equi-Lipschitz on $[t_0,+\infty)\times \mathscr P(K)$ with respect to $\phi$.
	\end{enumerate}
\end{Pro}

\begin{Pro}[\cite{CCSW2025}]\label{pro:sub-hjs}
The following statements are equivalent.
\begin{enumerate}
	\item $\phi$ is the subsolution to \eqref{eq:HJ_wk}; 
	\item $\phi(\nu)-\phi(\mu)\leqslant C^t(\mu,\nu)+c[0]t$ for all $t>0$ and $\mu,\nu\in\mathscr P_c(\R^m)$;
	\item For any $\mu\in\mathscr P_c(\R^m)$, $[0,+\infty)\ni t\mapsto P_t^-\phi(\mu)+c[0]t$ is non-decreasing;
	\item For any $\mu\in\mathscr P_c(\R^m)$, $[0,+\infty)\ni t\mapsto P_t^+\phi(\mu)-c[0]t$ is non-increasing.
\end{enumerate}
\end{Pro}

\subsection{Hamilton-Jacobi equations in $\mathscr P_c(\R^m)$}

Let \(\phi \in C(\mathbb{R}^m)\) be a function \((\kappa_1, \kappa_2)\)-Lipschitz in the large. As the value function of the Bolza problem, \(u(t, x) := T_t^- \phi(x)\) is the viscosity solution to equation \eqref{eq:HJe} with the initial condition \(\phi\). A natural question arises: whether \(U(t, \mu) := P_t^- \phi(\mu)\) will be the viscosity solution to some corresponding form of the Hamilton-Jacobi equations.

Assume that $H\in C(\R^{m}\times\R^{m})$ is a Tonelli Hamiltonian. We consider the following Hamilton-Jacobi equation: 
\begin{align}
	\partial_tU(t,\mu)+\int_{\R^m}H(x,\partial_\mu U(t,\mu)(x))\,d\mu =0,\tag{PHJ$_{e}$}\label{phje}
\end{align}
where $(t,\mu)\in\R^+\times\mathscr P_c(\R^m)$.

\begin{defn}[Viscosity solution of \eqref{phje}]\label{defn:viscosity}
We call the functional $U:\R^+\times\mathscr P_c(\R^m)$ a \textit{viscosity subsolution} of \eqref{phje}, if for every $(t,\mu)\in\R^+\times\mathscr P_c(\R^m)$, 
\begin{align}
	q+\int_{\R^m}H(x,\alpha(x))\,d\mu\leqslant0,\,\,\,\,\,\,\,\,\forall (q,\alpha)\in\partial^+U(t,\mu);\label{eq:phje-sub}
\end{align}
Similarly, $U:\R^+\times\mathscr P_c(\R^m)$ is called a \textit{viscosity supersolution} of \eqref{phje}, if for every $(t,\mu)\in\R^+\times\mathscr P_c(\R^m)$,
\begin{align}
	p+\int_{\R^m}H(x,\beta(x))\,d\mu\geqslant0,\,\,\,\,\,\,\,\,\forall (p,\beta)\in\partial^-U(t,\mu).\label{eq:phje-sup}
\end{align}
If $U$ is both a subsolution and supersolution to \eqref{phje}, we say $U$ is a \emph{viscosity solution} of \eqref{phje}. 
\end{defn}

In a more general and abstract frame, Gangbo, Nguyen and Tudorascu introduced concept of Hamilton-Jacobi equations in Wasserstein spaces in \cite{Gangbo_Nguyen_Tudorascu2008}, which is further studied in \cite{Gangbo_Tudorascu2019} later by Gangbo and Tudorascu. Although, \eqref{phje} can be regarded as a special case of their work, it is closely related to main contents of this paper.

\begin{defn}[\cite{Ambrosio_GigliNicola_Savare_book2008}, potential energy]\label{def:potential energy}
Given a Borel measurable function $\phi:\R^m\to[-\infty,+\infty]$. A functional of $\mathscr P(\R^m)$ induced by $\phi$
\begin{align*}
	\Phi(\mu):=\int_{\R^m}\phi(x)\,d\mu\qquad (\mu\in\mathscr P(\R^m))
\end{align*}
is called \textit{potential energy functional of} $\phi$. We denote $\Phi(\cdot)$ by $\phi(\cdot)$ for brevity if there is no confusion. 
\end{defn}

\begin{The}[\cite{CCSW2025}]\label{thm:phje-cauchy}
Let $\phi$ be lower semicontinuous and $(\kappa_1,\kappa_2)$-Lipschitz in the large. Then $U(t,\mu):=P_t^-\phi(\mu)$ is a viscosity solution of the following Cauchy problem of \eqref{phje}
\begin{align*}
	\left\{
		\begin{array}{ll}
			\partial_tU(t,\mu)+\int_{\R^m}H\left(x,\partial_\mu U(t,\mu)\right)\,d\mu=0,&(t,\mu)\in\R^+\times\mathscr P_c(\R^m),\\
			U(0,\mu)=\phi(\mu),& \mu\in \mathscr P_c(\R^m).
		\end{array}
	\right.
\end{align*}
\end{The}

Unlike the fundamental solution \( A_t(x,y) \), the dynamical cost functional \( C^t(\mu, \nu) \) can only possibly be a (strict) viscosity subsolution in general. This is because it can be regarded as a convex combination of \( A_t(x,y) \) based on measures \( \mu \) and \( \nu \). However, under the action of the random Lax-Oleinik operators, which corresponds to the convolution with the potential energy functional, it becomes a viscosity solution.

\begin{Rem}
If $\phi$ is upper semicontinuous and $(\kappa_1,\kappa_2)$-Lipschitz in the large, $V(t,\mu):=P_t^+\phi(\mu)$ is the viscosity solution to the following Cauchy problem
\begin{align*}
	\left\{
		\begin{array}{ll}
			\partial_tV(t,\mu)-\int_{\R^m}H\left(x,\partial_\mu V(t,\mu)\right)\,d\mu=0,&(t,\mu)\in\R^+\times\mathscr P_c(\R^m),\\
			V(0,\mu)=\phi(\mu),& \mu\in \mathscr P_c(\R^m).
		\end{array}
	\right.
\end{align*}	
\end{Rem}

\begin{Pro}[\cite{CCSW2025}]\label{pro:phjs}
The following statements are equivalent.
\begin{enumerate}
	\item $u$ is a weak KAM solution (viscosity solution) of \eqref{eq:HJ_wk} on $\T^m$;
	\item For any $t\geqslant 0$, $T_t^-u+c[0]t=u$;
	\item For any $t\geqslant 0$, $P_t^-u+c[0]t=u(\cdot)$;
	\item For any $t\geqslant 0$, $P_t^-\circ P_t^+u=u(\cdot)$.
\end{enumerate}	In this case, $u(\cdot)$ is the viscosity solution of the corresponding stationary type of Hamilton-Jacobi equation 
\begin{align}\tag{PHJ$_{s}$}\label{phjs}
	\int_{\T^m}H(x,\partial_\mu u(\mu)(x))\,d\mu=c[0],\,\,\,\,\mu\in\mathscr P(\T^m).
\end{align}
\end{Pro}

A simple fact is: if $u$ satisfies $u=T_t^-u+c[0]t$ for some $t\geqslant 0$, then $u=T_t^-\circ T_t^+u$ for such $t$. As a direct consequence of this fact, we have the following corollary.

\begin{Cor}[\cite{CCSW2025}]\label{cor:wkam-fixed}
Assume that $u\in\Lip(\R^m)\cap\SCL(\R^m)$ is a viscosity solution of \eqref{eq:HJ_wk}. Then for such $u$, we can also get $T_t^-\circ T_t^+u=u$ for arbitrary $t\geqslant 0$. As a direct consequence,
\begin{align*}
	P_t^-u+c[0]t=u(\cdot),\,\,P_t^-\circ P_t^+u=u(\cdot),\,\,\forall t\geqslant 0.
\end{align*}
\end{Cor}

\subsection{Propagation of singularities}

\begin{defn}[Calibrated curve]\label{def:u-measure-cali}
An absolutely continuous curve\footnote{For the definition on absolutely continuous curves of metric space, readers can see \cite[Definition 1.1.1]{Ambrosio_GigliNicola_Savare_book2008}.} $\{\mu_s\}_{s\in I}\subset  \mathscr P_c(\R^m)$ defined on some interval $I\subset\R$ is called a \emph{$(u(\cdot),C,c[0])$-calibrated curve}, if for any $a,b\in I$, $a\leqslant b$, we have 
\begin{align*}
	u(\mu_b)-u(\mu_a)=C^{b-a}(\mu_a,\mu_b)+c[0](b-a).
\end{align*}
\end{defn}

\begin{Rem}\label{rem:wass-cali-exist}
Suppose $u$ is a viscosity solution.  For $\mu\in\mathscr P_c(\R^m)$, there is a $(u(\cdot),C,c[0])$-calibrated curve ending at $\mu$. This is a direct consequence of the following fact mentioned in \cite{Fathi_book}: for each $x\in\R^m$, there exists a $(u,L,c[0])$-calibrated curve $\eta\in C^2((-\infty,0],\R^m)$ with $\eta(0)=x$ in $\R^m$. In other words, for any $a\leqslant b\leqslant 0$, 
\begin{align*}
u(\eta(b))-u(\eta(a))=\int_a^bL(\eta(s),\dot\eta(s))\,ds.	
\end{align*}
Define $\mu=\text{law}(x(\cdot))$. Then there exists a random curve $\xi=\xi(t)=\xi(t,\omega)$ such that $\xi(0,\omega)=x(\omega)$ and $\xi(\cdot,\omega)\in C^2((-\infty,0],\R^m)$ is a calibrated curve for each $\omega\in\Omega$. Suppose $\mu_s:=\text{law}(\xi(s))$, then $\{\mu_s\}_{s\in(-\infty,0]}$ is we need.
\end{Rem}

\begin{defn}[Cut point]\label{def:cut measure}
We call $\mu\in\mathscr P_c(\R^m)$ a \emph{cut point} of $u(\cdot)$, the potential energy functional induced by viscosity solution $u$, if all $(u(\cdot),C,c[0])$-calibrated curves ending at $\mu$ are unable to extend forward still as $(u(\cdot),C,c[0])$-calibrated curves. The set of all cut points of $u(\cdot)$ is denoted by $\mathscr C(u(\cdot))$, which is called the \emph{cut locus} of functional $u(\cdot)$.

In general, $\mathscr S(u(\cdot))\subset\mathscr C(u(\cdot))$.
\end{defn}

Inspired by the concept of cut time function, we introduce the \emph{cut time function of measure} with respect to the potential energy induced by viscosity solution $u$:
\begin{equation}\label{eq: cut time measure}
    \begin{split}
    	T_u(\mu):=&\,\sup\{t\geqslant 0:\exists\nu(\cdot)\in\mathrm{AC}([0,t];\mathscr P_c(\R^m)),\\
    	&\,\qquad\qquad \text{such that}\,u(\nu(t))=u(\mu)+C^t(\mu,\nu(t))+c[0]t\}.
    \end{split}
\end{equation}
Then $\mu\in\mathscr C(u(\cdot))$ if and only if $T_u(\mu)=0$. 

\begin{The}[\cite{CCSW2025}]\label{thm: measure cut time}
Assume that $u$ is the viscosity solution to \eqref{eq:HJ_wk} on $\R^m$ and $\mu\in\mathscr P_c(\R^m)$. Then
\begin{align*}
	T_u(\mu)&=\sup\{t\geqslant 0: P_t^-\circ P_t^+u(\mu)=P_t^+\circ P_t^-u(\mu)\}\\
		&=\sup\{t\geqslant 0:\mu(\{x\in\R^m:(T^-_t\circ T^+_t-T^+_t\circ T^-_t)u(x)=0\})=1\}\\
		&=\inf\{\tau_u(x):x\in\text{\rm supp}(\mu)\}.
\end{align*}
\end{The}

Following the concept of the Aubry set in Aubry-Mather theory or weak KAM theory, the \emph{Aubry set} of the potential energy function \( u(\cdot) \) is defined as the collection of all measures \(\mu\) for which there exists a \((u(\cdot), C, c[0])\)-calibrated curve \(\nu(\cdot): [0, +\infty) \to \mathscr{P}(\mathbb{T}^m)\) such that \(\nu(0) = \mu\). Consequently, a measure \(\mu\) is in the Aubry set of \( u(\cdot) \) if and only if \( T_u(\mu) = +\infty \).

\begin{Cor}[\cite{CCSW2025}]\label{cor:cut-prop}
Under the assumptions in Theorem \ref{thm: measure cut time},
\begin{enumerate}
	\item if $\supp(\mu)\cap\Cut(u)\neq\varnothing$, then $\mu\in\mathscr C(u(\cdot))$;
		\item if $\tau_u$ is continuous, then $\supp(\mu)\cap\Cut(u)\neq\varnothing$ if and only if $\mu\in\mathscr C(u(\cdot))$;
		\item $\mu$ is in the Aubry set of $u(\cdot)$ if and only if $\supp(\mu)\subset\mathcal I(u)$;
		\item if \(\mu = \text{\rm law}(Z)\) is the distribution of an \(\mathbb{R}^m\)-valued random variable \(Z : \Omega \to \mathbb{R}^m\), then \(\mu \in \mathscr{C}(u(\cdot))\) if and only if, for any \(\varepsilon > 0\),
		\begin{align*}
			\mathbb{P}\left( \{\omega \in \Omega : \tau_u(Z(\omega)) < \varepsilon \} \right) > 0.
		\end{align*}
\end{enumerate}
\end{Cor}

Applying the idea of intrinsic singular characteristics (see Section~\ref{subsec:intrinsic}) together with the methods of minimizing movement and maximal slope curves sketched below Theorem~\ref{thm:solution limit}, we obtain the following main results on the propagation of singularities for the viscosity solution of \eqref{phjs}.

\begin{The}[\cite{CCSW2025}]
Suppose \(\phi \in \SCL(\mathbb{T}^m)\), \(T > 0\), and \(\mu_0 \in \mathscr{P}(\mathbb{T}^m)\). Then there exists a Lipschitz curve \(\mu(\cdot): [0,+\infty) \to \mathscr{P}(\mathbb{T}^m)\) satisfying the following continuity equation:
\begin{equation}\label{eq:gene-sing-continuity}
\left\{
	\begin{array}{ll}
		\frac{d}{dt}\mu + \mathrm{div}(H_p(\cdot, \mathbf{p}_{\phi,H}^\#(\cdot)) \cdot \mu) = 0, \\
		\mu(0) = \mu_0,
	\end{array}
\right.
\end{equation}
where $\mathbf{p}_{\phi,H}^\#(\cdot)$ is defined in \eqref{eq:P^}. 

Moreover, if \(\phi\) in \eqref{eq:gene-sing-continuity} is a viscosity solution to \eqref{eq:HJ_wk} and \(\mu(\cdot): [0,+\infty)\to \mathbb{T}^m\) is the solution of \eqref{eq:gene-sing-continuity}, then
\begin{enumerate}
	\item If \(\mu_0(\overline{\text{Sing}\,(\phi)}) > 0\), then \([\mu(t)](\overline{\text{Sing}\,(\phi)}) > 0\) for all \(t \in [0, +\infty)\).
	\item If \(\mu_0 \in \mathscr{C}(\phi(\cdot))\), then \(\mu(t) \in \mathscr{C}(\phi(\cdot))\) for all \(t \in [0, +\infty)\).
\end{enumerate}
\end{The}

\begin{Rem}
Given \(\phi \in \SCL(\T^m)\), we define
\begin{align*}
	\overline{\mathscr S}(\phi(\cdot)) := \left\{ \mu \in \mathscr P(\T^m) : \mu(\overline{\text{Sing}\,(\phi)}) > 0 \right\}.
\end{align*}
Then, \(\overline{\mathscr S}(\phi(\cdot))\) is an \(F_\sigma\)-set in the sense of weak-$\ast$ topology. Suppose \(u\) is a viscosity solution of \eqref{eq:HJ_wk}. Then $\text{Sing}\,(u)\subset\text{Cut}\,(u)\subset\overline{\text{Sing}\,(u)}$. However, \(\mathscr C(u(\cdot))\) and \(\overline{\mathscr S}(u(\cdot))\) are not mutually contained in general, and we have \(\mathscr S(u(\cdot)) \subsetneq \mathscr C(u(\cdot)) \cap \overline{\mathscr S}(u(\cdot))\). 
\end{Rem}

\section{Out look and open problems}
\label{sec:open}

This field remains highly open, characterized by an essential irreversibility in the evolution of singularities within Hamiltonian systems that serves as a fundamental link between microscopic dynamics and macroscopic equilibrium. This field is deeply connected to several mathematical fields such as Symplectic Geometry, Optimal Transport and Mean Field Theory, and Stochastic Analysis. These connections suggest that the GHGF is not merely a technical tool for PDE analysis, but a central mechanism for understanding structural stability and dynamical irreversibility in complex systems. In light of these observations, we pose several open problems below.

\begin{Prob}\label{prob:uniqueness}
While the existence and stability of strict singular characteristics have been established (see Sect. \ref{subsec:wellposeness}), their fine dynamical properties under a general Tonelli Hamiltonian $H(x,p)$ remain a subject of ongoing research. For a general Tonelli system, the lack of a Riemannian inner product structure in the fibers poses significant challenges to a complete dynamical description. 
\begin{enumerate}
	\item Although the flow is known to be stable, the uniqueness of strict singular characteristics for general Tonelli Hamiltonians is not yet fully understood; only a limited number of cases have been treated so far (see, for instance, \cite{Cannarsa_Cheng2021b,Cheng_Hong2022a,Cheng_Hong2023}). We ask whether the strict convexity of $H$ alone is sufficient to prevent the branching of singular paths, or if the Finsler-like nature of the system allows for multiple energy-dissipating trajectories from the same initial singularity.
	\item A major distinction arises between the propagation of cut points and singular points. While the global propagation of cut points is already known (see Sect. \ref{subsec:prop_cut}), the behavior regarding the singular set \(\mathrm{Sing}(\phi)\) is less clear. If the Lagrangian is a Riemannian metric (\cite{ACNS2013}), a general mechanical Lagrangian (\cite{ACCM2026}), or even a magnetic mechanical Lagrangian (\cite{CCHW2025}), it is known that such curves stay within \(\mathrm{Sing}(\phi)\) globally. We seek to determine if this property is universal for Tonelli systems.
	\item This remains the central technical challenge for obtaining quantitative control over the flow. We ask whether along a strict singular characteristics $\gamma$, the forward energy evolution of $H(\gamma(t), \mathbf{p}^\#_{\phi,H}(\gamma(t)))$ satisfies certain Gronwall-type inequality like in \cite{ACCM2026}. Establishing such an estimate would provide the necessary tool to link the local variational optimality of the path to its global dynamical regularity.
	\item Does the left derivative $\dot{\gamma}^-(t)$ of a strict singular characteristic correspond to a specific \emph{min-max energy point} within the superdifferential $D^+\phi(x)$? Conversely, we ask whether every min-max energy point in $D^+\phi(x)$ induces a unique backward-in-time strict singular characteristic starting from $x$. Establishing this duality would link the topological properties of the momentum space to the selection of physical trajectories.
\end{enumerate}	
\end{Prob}

\begin{Prob}
Let $\{\mathbf{P}^t\}_{t \ge 0}$ be the semi-dynamical system induced by the generalized Hamilton gradient flow (see, for instance, Sect. \ref{subsec:GHGF}, \ref{subsec:Mather_max} and \ref{subsec:mass_transport}). While Mather measures are the classic invariant measures supported on the Aubry set $\mathcal{A}$, what other invariant or occupation measures exist beyond Mather measures within the cut locus $\text{Cut}(\phi)$? Furthermore, how do these singular invariant structures correspond to the elliptic regions of the underlying Hamiltonian system? This essentially relies on the technical point described in Problem \ref{prob:uniqueness} (3).
\end{Prob}

\begin{Prob}
Under what conditions is the cut locus $\text{Cut}(\phi)$ $\mathscr{H}^{n-1}$-rectifiable? Furthermore, what is the dynamical significance of this rectifiability regarding:
\begin{enumerate}
	\item The metric structure of the Aubry and Mather sets;
	\item The duality between the Hessian measure $[D^2\phi]$ and the Green bundles $(G_-, G_+)$;
	\item The regularity of the Cut time function $\tau(x)$ as a descriptor of singularity propagation.
\end{enumerate}
\end{Prob}

While the fact that $\text{Leb}(\text{Cut}(\phi)) = 0$ ensures the uniqueness of minimizing orbits for almost every starting point, it is insufficient to characterize the global stability of the flow. The $\mathscr{H}^{n-1}$-rectifiability provides the necessary geometric control to treat the cut locus as a propagating front. 

In the classical setting of the distance function $d_\Omega$ on a Finsler manifold, Li and Nirenberg \cite{Li_Nirenberg2005} established that $\text{Cut}(d_\Omega)$ is $\mathscr{H}^{n-1}$-rectifiable provided the boundary $\partial\Omega$ is of class $C^{2,1}$ (see also \cite{Mantegazza_Mennucci2003,Mennucci2004}). This regularity is optimal; rectifiability may fail if $\partial\Omega \in C^{2, \alpha}$ for $\alpha < 1$. While discussions on the Cauchy problem exist (e.g., \cite{Cannarsa_Sinestrari_book}), this problem remains one challenges in Weak KAM theory.

The rectifiability of $\text{Cut}(\phi)$ is a window into the global behavior of the Hamiltonian flow:
\begin{enumerate}
	\item The cut locus $\text{Cut}(\phi)$ can be viewed as the collision manifold of minimizing orbits. Dynamically, these orbits are backward asymptotic to the Aubry set $\mathcal{A}$. The $\mathscr{H}^{n-1}$-rectifiability of the cut locus acts as a geometric object that the basins of attraction (in the sense of the Lax-Oleinik semigroup) are tame. If rectifiability fails, it implies a fractalization of the action-minimizing phase space, suggesting that the Aubry set may possess a non-rectifiable, complex metric structure.
	\item There exists a profound duality between the Hessian measure $[D^2\phi]$ concentrated on $\text{Cut}(\phi)$ and the Green bundles $(G_-, G_+)$ supported on $\mathcal{A}$. The $\mathscr{H}^{n-1}$-rectifiability of $\text{Cut}(\phi)$ is fundamentally a statement about the controlled blow-up of the Riccati flow relative to the dynamical gap between $G_-$ and $G_+$. 
	\item From the perspective of Hamiltonian dynamics, $\mathscr{H}^{n-1}$-rectifiability implies that the shocks of the system are confined to a set of negligible $n$-dimensional volume. This is closely related to the Ma\~n\'e genericity conjecture (\cite{Mane1996,Contreras_Figalli_Rifford2015,Contreras2024}).
	\item The $\mathscr{H}^{n-1}$-rectifiability of $\text{Cut}(\phi)$ is intimately linked to the log-Lipschitz or BV regularity of the cut time function. If $\tau(x)$ is tame, the shocks of the Hamilton-Jacobi equation develop along well-defined fronts.
\end{enumerate}

\begin{Prob}
A long-standing problem in the geometric theory of HJ equations is the stability of the cut locus $\mathrm{Cut}(\phi)$ under perturbations of the potential $\phi$ or the Hamiltonian $H$. While the viscosity solution $\phi$ depends continuously on its data in the $C^0$-norm, its singular set can undergo drastic topological changes even under $C^\infty$-small perturbations. In this work, we raise the following questions regarding the Hausdorff stability of the cut locus:
\begin{enumerate}
	\item Consider the viscous solution $\phi^\epsilon$ or a sequence of smooth approximations $\phi_n \to \phi$. Does the transient cut locus $\mathrm{Cut}(\phi_n)$ converge to the limit $\mathrm{Cut}(\phi)$ in the Hausdorff sense? Given that $\mathrm{Cut}(\phi)$ is typically only rectifiable and not closed, we ask if the closure $\overline{\mathrm{Cut}(\phi)}$ possesses a form of structural stability under the GHGF.
	\item How does the singular set $\mathrm{Sing}(\phi)$ respond to a change of $H$ in the cohomology class $[\alpha] \in H^1(M, \mathbb{R})$? We investigate whether certain critical classes trigger topological phase transitions in the cut locus, such as the sudden connection of isolated singular components.
	\item For a fixed Hamiltonian $H$, what is the relationship between the singular sets of different viscosity solutions $\phi_1, \phi_2$? We seek to determine if there exist universal singular structures dictated solely by the Hamiltonian dynamics of $H$.
	\item The cut locus acts as a dynamical skeleton for the manifold. We investigate whether the $\mathscr{H}^{n-1}$-measure of the cut locus is stable, or if ghost singularities can vanish abruptly in the limit. Specifically, we explore whether the convergence of the Hessian measures $[D^2\phi_n] \to [D^2\phi]$ in the sense of distributions can be strengthened to provide geometric convergence of their supports in the Hausdorff metric.
\end{enumerate}	
\end{Prob}

\begin{Prob}
Canonical contact Hamiltonian systems and the associated contact-type Hamilton-Jacobi equations have become an active area of research over the past decades. The analysis relies on an implicit variational principle developed in \cite{Wang_Wang_Yan2017,Wang_Wang_Yan2019_1}, or on Herglotz' generalized variational principle (\cite{CCWY2019,CCJWY2020}) in a more general setting, or on the special case of the discounted Hamilton-Jacobi equation (see, for instance, \cite{DFIZ2016,Maro_Sorrentino2017}).
\begin{enumerate}
	\item In contact Hamiltonian systems, the presence of an internal dissipation function $\lambda$ (derived from Herglotz's generalized variational principle) breaks the conservative symplectic structure. To generalize the GHGF framework to this setting, we must reformulate the energy dissipation inequality to capture the non-conservative dissipation rate along characteristics. We need to derive the GHGF for contact-type Hamiltonian systems.
	\item In the vanishing discount or vanishing contact problem, the Hausdorff stability of the cut locus with respect to the internal dissipation function $\lambda$ depends critically on the sign of $\lambda$ as $\lambda\to 0$.  For the three regimes $\lambda\to0^+$ (positive dissipation), $\lambda\to0^-$ (negative dissipation, i.e., anti-damping), and the hybrid direction (e.g., $\lambda$ taking both signs or depending on $x$), the following questions arise: does the cut locus converge in the Hausdorff sense to the cut locus of the limiting Hamiltonian system, or does its Hausdorff limit contain additional singularities? Alternatively, can the cut locus exhibit oscillatory behavior so that only weak convergence (e.g., $\Gamma$-limit or Hausdorff lower/upper semicontinuity) can be expected in general?
\end{enumerate}
\end{Prob}

\begin{Prob}
Let $\{T_t^- u_0\}_{t \ge 0}$ be the Lax-Oleinik semigroup associated with a Tonelli Hamiltonian $H$ with Ma\~n\'e's critical value $c[H]=0$. As $t \to \infty$, $T_t^- u_0$ converges uniformly to a viscosity solution $u_\infty$ of \eqref{eq:HJ_wk}. How does the geometric structure of the cut locus $\text{Cut}(T_t^- u_0)$ and/or $\text{Cut}(u_\infty)$ quantitatively govern the convergence rate $\|T_t^- u_0 - u_\infty\|_{C^0}$?
\end{Prob}

\begin{Prob}
To broaden the applicability of the GHGF, the theory must be extended to systems where classical Tonelli assumptions fail.
\begin{enumerate}
    \item When $H(x,p)$ is non-convex, variational methods lose their traditional power. Can we reconstruct the GHGF theory from a purely \emph{symplectic geometry} perspective, treating singularities as caustics or Lagrangian intersections?
    \item For Hamiltonians with low regularity (e.g., non-$C^2$ or discontinuous), does the GHGF exist as a \emph{measure-valued flow}? In this framework, the evolution of singularities would be described as a concentration process of probability measures in the configuration space.
\end{enumerate}
\end{Prob}

\begin{Prob}
Consider the viscous stationary Hamilton-Jacobi equation on a compact manifold $M$,
\begin{align*}
	H(x,Du^\epsilon(x)) - \epsilon \Delta u^\epsilon(x) = c[H]
\end{align*}
with $c[H]=0$. Let $X^\epsilon_t$ be the \emph{forward stochastic maximal slope curve} satisfying the stochastic differential equation
\begin{align*}
	dX^\epsilon_t = H_p(X^\epsilon_t,Du^\epsilon(X^\epsilon_t)) dt + \sqrt{2\epsilon} dW_t.
\end{align*}
As $\epsilon \to 0$, $u^\epsilon$ converges uniformly to the viscosity solution $u$ of \eqref{eq:HJ_wk}. When the strict singular characteristics are not unique, what is the asymptotic behavior of the probability measures induced by $X^\epsilon_t$? In particular, does the vanishing noise limit select a specific statistical distribution among the multiple possible singular paths?
\end{Prob}

\begin{Prob}
Let the setting be that of \cite{CCSW2025}. Consider the nonlinear free energy functional defined on $\mathscr{P}(\T^m)$
\begin{align*}
	\Phi_\beta(\nu)=\int_{\mathbb{R}^m}\phi\,d\nu + \beta\,\text{Ent}\,(\nu\mid\text{ref}),\qquad \beta\ge 0,
\end{align*}
where $\text{Ent}(\nu\mid\text{ref})$ is the relative entropy with respect to a reference measure $\text{ref}$. Define the nonlinear Lax–Oleinik operator as
\begin{align*}
	\mathcal{T}_t^-\Phi_\beta(\mu):=\inf_{\nu\in\mathscr{P}_c(\T^m)}\bigl\{\Phi_\beta(\nu)+C^t(\nu,\mu)\bigr\},
\end{align*}
and set $U_\beta(t,\mu):=\mathcal{T}_t^-\Phi_\beta(\mu)$.
\begin{enumerate}
	\item For every $\beta\ge 0$, does $U_\beta$ satisfy the dynamic programming principle and the \emph{master equation} in the viscosity sense of Definition~\ref{defn:viscosity}? Furthermore, what is the exact class of admissible initial functionals (containing both $\Phi_0$ and $\Phi_\beta$ for $\beta>0$) for which the results of \cite{CCSW2025} can be extended?
	\item For any $\beta>0$, is the minimizer $\nu_\beta$ in the definition of $U_\beta(t,\mu)$ absolutely continuous with respect to $\text{ref}$? Does this ensure that $U_\beta(t,\cdot)$ is everywhere Lions-differentiable such that $\text{Sing}\,(U_\beta(t,\cdot))=\varnothing$ for all $t>0$, in contrast to the singular case when $\beta=0$?
	\item As $\beta \to 0^+$, does the functional converge pointwise (or uniformly on compact sets) to $U_0(t,\mu)$? Can this limit be interpreted as an infinite-dimensional vanishing-viscosity method where the entropic term regularizes the singularities?
\end{enumerate}
\end{Prob}
 
%
%
%
%
%
%
%
%

\bibliography{mybib.bib}
\bibliographystyle{plain}


%
%
%
%
%
%
%
%
%

\end{document}